\documentclass{amsart}
\usepackage{latexsym}
\usepackage{amssymb}
\usepackage{epsfig}
\usepackage{amsthm}
\usepackage{latexsym}
\usepackage{longtable}
\usepackage{epsfig}
\usepackage{amsmath}
\usepackage{hhline}
\usepackage{color}
\newtheorem{theorem}{Theorem}[section]

\newtheorem{lm}[theorem]{Lemma}
\newtheorem{tr}[theorem]{Theorem}
\newtheorem{cor}[theorem]{Corollary}

\newtheorem{rem}[theorem]{Remark}
\newtheorem{pr}[theorem]{Proposition}

\begin{document}
\title[Donkin-Koppinen filtration and generalized Schur superalgebras]{Donkin-Koppinen filtration for $GL(m|n)$ and generalized Schur superalgebras}

\author{F.Marko and A.N.~Zubkov}
\address{The Pennsylvania State University, 76 University Drive, Hazleton, 18202 PA, USA}
\email{fxm13@psu.edu}
\address{Department of Mathematical Sciences, UAEU, Al-Ain, United Arab Emirates; \linebreak Sobolev Institute of Mathematics, Omsk Branch, Pevtzova 13, 644043 Omsk, Russia}
\email{a.zubkov@yahoo.com}

\thanks{The work was supported by the UAEU grant G00003324.}

\begin{abstract}
The paper contains results that characterize the Donkin-Koppinen filtration of the coordinate superalgebra $K[G]$ of the general linear supergroup $G=GL(m|n)$ by its
subsupermodules $C_{\Gamma}=O_{\Gamma}(K[G])$. Here, 
the supermodule $C_{\Gamma}$ is the largest subsupermodule of $K[G]$ whose composition factors are irreducible supermodules of highest weight $\lambda$, where $\lambda$ belongs to a finitely-generated ideal $\Gamma$ of the poset $X(T)^+$ of dominant weights of $G$. A decomposition of $G$ as a product of subsuperschemes $U^-\times G_{ev}\times U^+$ induces a superalgebra isomorphism $\phi^* : 
K[U^-]\otimes K[G_{ev}]\otimes K[U^+]\simeq K[G]$. We show that $C_{\Gamma}=\phi^*(K[U^-]\otimes M_{\Gamma}\otimes K[U^+])$, where $M_{\Gamma}=O_{\Gamma}(K[G_{ev}])$. Using the basis of the module $M_{\Gamma}$, given by generalized bideterminants, we describe a basis of $C_{\Gamma}$.

Since each $C_{\Gamma}$ is a subsupercoalgebra of $K[G]$, its dual $C_{\Gamma}^*=S_{\Gamma}$ is a (pseudocompact) superalgebra, called the generalized Schur superalgebra. There is a natural superalgebra morphism  $\pi_{\Gamma}:Dist(G)\to S_{\Gamma}$ such that the image of the distribution algebra $Dist(G)$ is dense in $S_{\Gamma}$.  For the ideal $X(T)^+_{l}$, of all weights of fixed length $l$, the generators of the kernel of $\pi_{X(T)^+_{l}}$ are described.
\end{abstract}

\maketitle

\section*{Introduction}

We start with a history of the subject. Let $G$ be a reductive group and $\Gamma$ be a finitely-generated ideal of (dominant) weights of $G$. The coordinate algebra $K[G]$ of $G$ can be regarded as a left $G$-module via left and right regular representations $\rho_l$ and $\rho_r$ respectively (cf. part I, 2.8 of \cite{jan}). Considering $K[G]$ as a left $G$-module with respect to $\rho_r$, let $M_{\Gamma}=O_{\Gamma}(K[G])$ be the largest $G$-submodule of $K[G]$ whose composition factors are irreducible modules of highest weight $\lambda\in\Gamma$. Then $M_{\Gamma}$ is a subcoalgebra of $K[G]$, and it is a $G\times G$-submodule of $K[G]$ with respect to $\rho_l\times\rho_r$. Donkin (\cite{don, don2}) and Koppinen (\cite{kop}) have proved that, for any maximal element $\lambda$ of $\Gamma$, the factormodule $M_{\Gamma}/M_{\Gamma\setminus\{\lambda\}}$ is isomorphic to the tensor product of the contragredient dual of Weyl module $V(\lambda)$ and the induced module $H^0(\lambda)$.
The filtration of $K[G]$ by the submodules $M_{\Gamma}$ is called a \emph{Donkin-Koppinen filtration}.

Donkin defined \emph{generalized Schur algebras} as $S_{\Gamma}=M_{\Gamma}^*$ and proved that each $S_{\Gamma}$ is a finite-dimensional quasi-hereditary algebra (cf. \cite{don2, jan}). 
The description of the generalized Schur algebras, in terms of generators and defining relations, is known only in some particular cases. For example, if $G$ is the general linear group $GL(m)$, and $\Gamma$ is the ideal of all polynomial weights of $GL(m)$ of fixed length $r$, then $S_{\Gamma}$ is the classical Schur algebra $S(m, r)$ (cf. \cite{martin}). 
Over ground fields of characteristic zero, the presentation of $S(m,r)$ by generators and defining relations was given in 
\cite{dotygiaq}. This result was extended to rational Schur algebras, which are also generalized Schur algebras, in \cite{dippdoty}.
 
Let $G=GL(m|n)$ be the general linear supergroup and $K[G]$ its coordinate superalgebra. As above, the superalgebra $K[G]$ has a natural structure of a left $G$-supermodule via the left and right regular representations $\rho_l$ and $\rho_r$, respectively. 
Fix an ideal $\Gamma$ of (dominant) weights of $G$. Considering $K[G]$ as a left $G$-supermodule with respect to $\rho_r$, let $O_{\Gamma}(K[G])$ denote the 
union of all finite-dimensional subsupermodules of $K[G]$ which have a composition series with irreducible factors of highest weight $\lambda\in \Gamma$. If $\Gamma$ is a finitely generated ideal, then Theorem 6.1 of \cite{sz} shows that $C_{\Gamma}=O_{\Gamma}(K[G])$ is a $G\times G$-subsupermodule of $K[G]$ with respect to $\rho_l\times\rho_r$. Moreover, for every maximal element $\lambda\in\Gamma$, the factor $C_{\Gamma}/C_{\Gamma\setminus\{\lambda\}}$ is isomorphic to the tensor product of the contragredient dual of the Weyl supermodule $V_-(\lambda)$ and the induced supermodule $H_-^0(\lambda)$. The filtration of $K[G]$ by subsupermodules $C_{\Gamma}$ is also called a \emph{Donkin-Koppinen filtration}.
These statements are valid even when $\Gamma$ is not finitely generated. In this case,  $C_{\Gamma} =\varinjlim C_{\Gamma'}$, where
$\Gamma'$ runs over finitely generated subideals of $\Gamma$. For any maximal element $\lambda\in\Gamma$, there is
$C_{\Gamma}/C_{\Gamma\setminus\{\lambda\}}=\varinjlim C_{\Gamma'}/C_{\Gamma'\setminus\{\lambda\}}$, where $\Gamma'$ runs over all finitely generated subideals containing 
$\lambda$. However, the factor $C_{\Gamma'}/C_{\Gamma'\setminus\{\lambda\}}$ does not depend on $\Gamma'$, i.e. all factors $C_{\Gamma'}/C_{\Gamma'\setminus\{\lambda\}}$ are isomorphic to each other.

Throughout the paper, we assume that the characteristic of the ground field $K$ is $p>2$, with the only exception that in Section 14 we assume that the characteristic of $K$ is zero.
The first part of this paper is devoted to further investigation of the Donkin-Koppinen filtration for $G=GL(m|n)$. The most important results are Theorems \ref{DK?} and \ref{basis of a DK factor}, which provide a $K$-basis of each member of the Donkin-Koppinen filtration.

As in the purely even case, a subsuperspace $C_{\Gamma}$ is a subsupercoalgebra of $K[G]$ for an arbitrary (not necessarily finitely-generated) ideal $\Gamma$. Thus $S_{\Gamma}=C_{\Gamma}^*$ is an infinite-dimensional (pseudocompact) superalgebra, called a \emph{generalized Schur superalgebra}.

The second part of the paper is devoted to results about generalized Schur superalgebras $S_{\Gamma}$ for an arbitrary ideal $\Gamma$ of weights. In Proposition \ref{quasi-hereditariness}, we show that if $\Gamma$ is finitely generated, then  $S_{\Gamma}$ is an ascending quasi-hereditary algebra in the sense of \cite{markozub}.
There is a natural superalgebra morphism $\pi_{\Gamma}: Dist(G) \to S_{\Gamma}$ given by the restriction. We show that the image $\pi_{\Gamma}(Dist(G))$ is dense in the pseudocompact topology on $S_{\Gamma}$. Therefore, to describe $S_{\Gamma}$, we need to understand the induced topology on $Dist(G)$ and determine the kernel of $\pi_{\Gamma}$.
We obtain a characterization of the induced topology in Lemma \ref{induced topology}, and a characterization of the elements of $\ker\pi_{\Gamma}$ in Lemma \ref{ker of a morphism}.  In Proposition \ref{lm3.3} we find generators of $ker(\pi_{\Gamma})\cap Dist(T)$, where $T$ is the torus of $G$.
Finally, for the ideal $X(T)^+_{l}$ of all weights of fixed length $l$, we describe the kernel of 
$\pi_{X(T)^+_{l}}$ in  Theorem \ref{the kernel} for odd characteristics, and in Theorem \ref{kernel0} for characteristic zero.

The structure of the paper is as follows. In Section 1 through 4, we overview relevant definitions and results concerning supergroups and their representations, pseudocompact algebras, distribution superalgebra, and standard (Weyl) and costandard (induced) supermodules over $GL(m|n)$. 
In Section 5, we prove that the Donkin-Koppinen filtration in Theorem 6.1 of \cite{sz} coincides with the above filtration consisting of subsupermodules $C_{\Gamma}$. We also show that, for every maximal element $\lambda\in\Gamma$, a $G\times G$-supermodule $C_{\Gamma}$ is congruent to a certain finite-dimensional subsupermodule modulo $C_{\Gamma\setminus\{\lambda\}}$ (see Proposition \ref{likeDonkin} and Corollary \ref{Donkin-Koppinenrealization}).
In Section 6, we give an explicit description of Donkin-Koppinen filtration of $G$ using the decomposition of $G$ as a product of subsuperschemes $U^-\times G_{ev}\times U^+$.
In Section 7, we give a $K$-basis of the factors of Donkin-Koppinen filtration of $K[G]$ using combinatorial tools involving bitableaux, and generalized bideterminants.
In Section 8, we discuss generalized Schur superalgebras $S_{\Gamma}$ corresponding to an ideal of weights $\Gamma$.
In Section 9, we show that for finitely generated $\Gamma$, the superalgebra $S_{\Gamma}$ is an ascending quasi-hereditary algebra.
In Section 10, we use a natural superalgebra morphism $\pi_{\Gamma}:Dist(G)\to S_{\Gamma}$ to describe $S_{\Gamma}$ as completion of $\pi_{\Gamma}(Dist(G))$ in the pseudocompact topology, and characterize the induced topology on $Dist(G)$. 
In Section 11, we compute generators of the intersection of $\ker \pi_{\Gamma}$ and the distribution algebra $Dist(T)$ of the torus $T$ of $G$.
In Section 12, we derive commutation formulae for the generators of $Dist(T)$ and the remaining generators of $Dist(G)$.
In Section 13, we describe generators of the kernel of the morphism $\pi_{X(T)^+_{l}}$ by proving that the kernel of the morphism $\pi_{X(T)^+_{l}}$ is generated by
$\ker\pi_{X(T)^+_{l}}\cap Dist(T)$.
Finally, in Section 14, the same task is accomplished over ground fields of characteristic zero.

\section{Supergroups and their representations}

For the content of this section, we refer to \cite{brunkuj, zub1, zub3}.

A $\mathbb{Z}_2$-graded $K$-space is called a \emph{superspace}. If $W$ is a superspace with grading $W=W_0\oplus W_1$, then the parity function $(W_0\cup W_1)\setminus\{0\}\to \mathbb{Z}_2$ is given by $w\mapsto |w|=i$, where $w\in W_i$ and $i\in\mathbb{Z}_2=\{0, 1\}$. 
$\mathsf{SVect}_K$ denotes the category of superspaces with graded (parity preserving) morphisms. The category $\mathsf{SVect}_K$ is abelian, and it is also a tensor category for the braiding
$t_{U, W} : U\otimes W\simeq W\otimes U$ given by 
\[u\otimes w\mapsto (-1)^{|u||w|}w\otimes u\] for $w\in W$ and $ u\in U$.

An (associative)  superalgebra is an algebra object in $\mathsf{SVect}_K$.  The superalgebras form a tensor subcategory of $\mathsf{SVect}_K$ when we define the superalgebra structure on $A\otimes B$ by
\[(a\otimes b)(c\otimes d)=(-1)^{|b||c|}ac\otimes b\]
for $a, c\in A$ and $b, d\in B$.  

Let $A$ be a superalgebra. Let $Smod-A$ (and $A-Smod$, respectively) denote the category of right (and left, respectively) $A$-supermodules considered with graded morphisms.

A superalgebra $A$ is called \emph{supercommutative}, if $ab=(-1)^{|a||b|}ba$ for each homogeneous elements $a, b\in A$.
Let $\mathsf{SAlg}_K$ denote the subcategory of $\mathsf{SVect}_K$ consisting of all (super)commutative superalgebras with graded morphisms.

Let $C$ be a supercoalgebra, i.e., $C$ is a coalgebra object in $\mathsf{SVect}_K$. In what follows, $\Delta_C : C\to C\otimes C$ and $\epsilon_C : C\to K$ denote the comultiplication and counit of $C$, respectively. 

Let $Scomod^C$ (and $^C Scomod$, respectively) denote the category of right (and left, respectively) $C$-supercomodules with graded morphism of $C$-comodules.
If $M$ is left (or right, respectively) $C$-supercomodule, then its comodule map $M\to C\otimes M$ (or $M\to M\otimes C$, respectively) is denoted by $\tau_M$.

An affine supergroup is a representable functor $G : \mathsf{SAlg}_K \to \mathsf{Gr}$, where $\mathsf{Gr}$ is the category of groups.
This means that there is a Hopf superalgebra $A$ such that
\[G(B)=\mathrm{Hom}_{\mathsf{SAlg}_K}(A, B)\]
for $B\in\mathsf{SAlg}_K$.
The Hopf superalgebra $A$ is called a \emph{coordinate superalgebra} of $G$ and is denoted by $K[G]$.
If $K[G]$ is finitely generated, then $G$ is called an \emph{algebraic supergroup}.

A morphism of affine supergroups $\pi : G\to H$ is uniquely defined by the dual morphism $\pi^* : K[H]\to K[G]$ of Hopf superalgebras. 

A supergroup $H$ of $G$ is a \emph{closed supersubgroup} of $G$ if $K[H]\simeq K[G]/I_H$, where
$I_H$ is a Hopf subsuperideal of $K[G]$. The supergroup $H$ is a group subfunctor of $G$.
For example, if $\pi : G\to H$ is a supergroup morphism, then its kernel $\ker\pi$ is a closed normal subsupergroup of $G$, defined by the superideal $K[G]\pi^* (K[H]^+)$, where $K[H]^+=\ker\epsilon_H$. 

The \emph{largest purely even} subsupergroup $G_{ev}$ of $G$ corresponds to the Hopf superideal $K[G]K[G]_1$.

The category of left (and right, respectively) $G$-supermodules $G-Smod$ (and $Smod-G$, respectively) coincides with the category $Scomod^{K[G]}$ (and $^{K[G]} Scomod$, respectively).

In what follows, we consider all $G$-supermodules as left $G$-supermodules, unless stated otherwise.

For example, one can define two structures of a left $G$-supermodule on $K[G]$. The first one, denoted by $\rho_r$,  coincides with $\Delta_G$, and it is called the \emph{right regular representation} of $G$ on $K[G]$. The second one, denoted by $\rho_l$, is given by $t_{K[G], K[G]}(s_G\otimes\mathrm{id}_{K[G]})\Delta_G$, and it is called the \emph{left regular representation} of $G$ on $K[G]$ (cf. I.2.7-2.8 of \cite{jan}). 

Let $\sigma$ be an anti-automorphism of the Hopf superalgebra $K[G]$, i.e., $\sigma$ induces an automorphism of the  superalgebra structure of $K[G]$ and is an anti-automorphism of the supercoalgebra structure of $K[G]$, simultaneously. 

If $M$ is a finite-dimensional $G$-supermodule, then we can define its $\sigma$-dual $M^{<\sigma>}$ as follows (cf. \cite{zub1}). Fix a homogeneous basis of $M$ consisting of elements $m_i$ for $1\leq i\leq t$. If 
\[\tau_M(m_i) =\sum_{1\leq k \leq t} m_k\otimes f_{ki},\] where $f_{ki}\in K[G]$ and their parities as given as $|f_{ki}| =|m_i| +|m_k| \pmod 2$, then the $G$-supermodule $M^{<\sigma>}$ has a basis consisting of elements $m_i^{<\sigma>}$
such that
\[\tau_{M^{<\sigma>}}(m_i^{<\sigma>})=\sum_{1\leq k\leq t} (-1)^{|m_k|(|m_k|+|m_i|)}m_k^{<\sigma>}\otimes \sigma(f_{ik}).\]
If $\sigma^2=\mathrm{id}_{K[G]}$, then $M\to M^{<\sigma>}$ is a self-duality of the full subcategory of all finite dimensional $G$-supermodules.

Let $P$ and $P'$ be subsupergroups of $G$. We say that they are $\sigma$-{\it connected} if $\sigma(I_P)=I_{P'}$. If $P$ and $P'$ are $\sigma$-connected, then $\sigma$ induces an anti-isomorphism $K[P]\to K[P']$. Moreover, the correspondence $M\mapsto M^{<\sigma>}$ induces a contravariant functor from the category of all finite-dimensional $P$-supermodules to the category of all finite-dimensional $P'$-supermodules. Additionally, if $\sigma^2=\mathrm{id}_{K[G]}$, then this functor is a duality.  

Consider $K[G]$ as a (left) $G\times G$-supermodule via $\rho_l\times\rho_r$. In other words, 
$K[G]$ is regarded as a right $K[G]\otimes K[G]$-supercomodule via 
\[f\mapsto \sum (-1)^{|f_1||f_2|}f_2\otimes s_G(f_1)\otimes f_3,\]
where $f\in K[G]$ and 
\[(\Delta_G\otimes\mathrm{id})\Delta_G (f)=(\mathrm{id}\otimes\Delta_G)\Delta_G (f)=\sum f_1\otimes f_2\otimes f_3.\]

Let $M$ be a finite dimensional $G$-supermodule. 
\begin{lm}\label{canonicalmap}
The linear map $\rho_M : M^*\otimes M\to K[G]$, given by
\[\alpha\otimes m\mapsto\sum \alpha(m_1)f_2, \]
where $\alpha\in M^*, m\in M$, and $\tau_M(m)=\sum m_1\otimes f_2,$
is a morphism of $G\times G$-supermodules. Moreover, $\mathrm{Im}(\rho_M)=\mathrm{cf}(M)$.
\end{lm}
\begin{proof}
Let $m_1, \ldots, m_t$ form a homogeneous basis of a superspace $M$. Then the supercomodule structure  
of $M$ is given by
\[\tau_M(m_i)=\sum_{1\leq j\leq t} m_j\otimes f_{ji}\]
for $1\leq i\leq t$, where the elements $f_{ji}$ satisfy
\[\Delta_G(f_{ji})=\sum_{1\leq k\leq t}f_{jk}\otimes f_{ki}\]
for $1\leq j, i\leq t$.

Let $m_1^*, \ldots, m_t^*$ be the dual basis of $M^*$. Then the supercomodule structure of $M^*$ is given by
\[\tau_{M^*}(m_i^*)=\sum (-1)^{|m_j|(|m_i|+|m_j|)}m_j^*\otimes s_G(f_{ij})\]
for $1\leq i\leq t$.
A routine calculation shows that the map $\rho_M\tau_{M^*\otimes M}$ sends any $m_i^*\otimes m_j$ to
\[\sum_{1\leq k, l\leq t}(-1)^{(|f_{ik}|)(|f_{kl}|)}f_{kl}\otimes s_G(f_{ik})\otimes f_{lj}.\]
which coincides with the action of $\rho_l\otimes \rho_l$ on the element $f_{ij}$.

Since $\rho_M(m_i^*\otimes m_j)=f_{ij}$, the last claim follows.
\end{proof}
\begin{lm}\label{another formula}
If we identify $M^*\otimes M$ with $\mathrm{End}_K(M)$, then for any $g\in G(A), A\in\mathsf{SAlg}_K$, and 
$\phi\in\mathrm{End}_K(M)$, there is \[g(\rho_M(\phi))=\mathrm{str}(g\circ\phi_0)+\mathrm{tr}(g\circ\phi_1).\] In particular, we have $\rho_M(\phi)= \mathrm{str}(\mathrm{id}_{K[G]}\circ\phi_0)+\mathrm{tr}(\mathrm{id}_{K[G]}\circ\phi_1)$.
\end{lm}
\begin{proof}
Since both sides of the above equality are linear in $\phi$, it is enough to check the case $\phi=m_i^*\otimes m_j$ only.
For this $\phi$, we have $\phi(m_k)=(-1)^{|m_k||m_j|}\delta_{i k}m_j$ for every $1\leq k\leq t$.
Then
\[(g\circ\phi)(m_k)=(-1)^{|m_k||m_j|}\delta_{i k}\sum_{1\leq l\leq t} m_l\otimes g_{lj},\]
and 
\[(g\circ\phi)_{ki}=\{\begin{array}{cc}
0 & \text{ if }  k\neq i\\
(-1)^{|m_i||m_j|}g_{ij} & \text{ if } k=i
\end{array}.
\]
Thus the first equality follows. The second equality is evident.
\end{proof}

Let $V$ be a superspace of superdimension $m|n$, i.e., $\dim V_0=m, \dim V_1=n$.
Then the functor
\[A\to \mathrm{End}_A(V\otimes A)^{\times}_0\]
is an algebraic supergroup, which is called a \emph{general linear supergroup}, and it is denoted by
$GL(V)$ or by $GL(m|n)$. Let $X$ denote the \emph{generic matrix} $(x_{ij})_{1\leq i, j\leq m+n}$. The matrix $X$ can be represented as
\[X=\left(\begin{array}{cc}
X_{11} & X_{12} \\
X_{21} & X_{22}
\end{array}\right),\]
where
\[X_{11}=(x_{ij})_{1\leq i, j\leq m}, \ X_{12}=(x_{ij})_{1\leq i\leq m < j\leq m+n},\]
\[X_{21}=(x_{ij})_{1\leq j\leq m < i\leq m+n}, \ X_{22}=(x_{ij})_{m< i, j\leq m+n}.\] 
Denote $D=\det(X_{11})\det(X_{22})$ and assign the parities of the variables $x_{ij}$ in such a way that 
$|x_{ij}|=0$ if and only if $1\leq i, j\leq m$ or $m< i, j\leq m+n$, and $|x_{ij}|=1$ otherwise.

Then the general linear supergroup $GL(m|n)$ is represented by the Hopf superalgebra
\[K[GL(m|n)]=K[x_{ij}\mid 1\leq i, j\leq m+n]_{D} ,\]
where the comultiplication and the counit of $K[GL(m|n)]$ are defined by
\[\Delta_{GL(m|n)}(x_{ij})=\sum_{1\leq k\leq m+n} x_{ik}\otimes x_{kj}\]
and 
\[ \epsilon_{GL(m|n)}(x_{ij})=\delta_{ij}.\]

The superspace $V$ has a natural $GL(m|n)$-supermodule structure given by the map 
\[\tau_V : v_i\mapsto\sum_{1\leq k\leq m+n} v_k\otimes x_{ki},\]
where the vectors $v_1, \ldots , v_{m+n}$ form a basis of $V$ such that $|v_i|=0$ if and only if $1\leq i\leq m$, and $|v_i|=1$ otherwise. We call $V$ the natural $GL(m|n)$-supermodule.

\section{Pseudocompact superalgebras}

Let $C$ be a supercoalgebra. The dual superspace $C^*$ has a natural structure of a pseudocompact superalgebra given by the multiplication
\[\phi\psi(c)=\sum (-1)^{|\psi||c_1|}\phi(c_1)\psi(c_2),\]
where $\phi, \psi\in C^*, c\in C$  and $\Delta_C(c)=\sum c_1\otimes c_2$.

For the definition of pseudocompact superalgebra, we refer the reader to \cite{markozub}.

A basis of neighborhoods at zero consists of two-sided ideals 
\[D^{\perp}=\{\phi\in C^*\mid \phi(D)=0\}\] for all finite-dimensional subcoalgebras $D$ of $C$. 
Since every such subcoalgebra $D$ is contained in a finite-dimensional subsupercoalgebra $\tilde{D}$, we can assume that the basis of neighborhoods at zero consists of two-sided superideals $\tilde{D}^{\perp}$. 

Repeating arguments from the proof of Theorem 3.6 in \cite{sims}, we can show that the functor $C\mapsto C^*$ is a duality between the category of supercoalgebras and the category of pseudocompact superalgebras (both considered with graded morphisms). 
The inverse functor to $C\mapsto C^*$ is given by
\[R\mapsto R^{\circ}=\varinjlim_{I} (R/I)^*, \] 
where $I$ runs over all two-sided open superideals of $R$. If $A$ is a finite-dimensional superalgebra, then the supercoalgebra structure of $A^*$ is uniquely defined by the rule 
\[\phi(ab)=\sum (-1)^{|\phi_2||a|}\phi_1(a)\phi_2(b),\]
where $a, b\in A, \phi\in A^*$ and $\Delta_{A^*}(\phi)=\sum\phi_1\otimes\phi_2$ (see Section 1 of \cite{zub2}).
If $R=C^*$, then the superspace embedding $C\to R^*$ given by $c\mapsto\hat{c}$, where 
$\hat{c}(r)=r(c)$ for $c\in C$ and $r\in R$, induces the supercoalgebra isomorphism $C\to  R^{\circ}$.
\begin{lm}\label{correspondence}
Let $C$ be a supercoalgebra. The map $m:D\mapsto D^{\perp}$ is an inclusions reversing one-to-one correspondence between subsupercoalgebras $D$ of $C$ and closed two-sided superideals of $R=C^*$.
\end{lm}
\begin{proof}
The two-sided superideal $D^{\perp}$ of $R$ coincides with the kernel of the morphism $C^*\to D^*$ of pseudocompact superalgebras, hence it is closed. Therefore, the map $m$ is well-defined. Its inverse map is given by 
\[J\mapsto (R/J)^{\circ}\subseteq R^{\circ}\simeq C,\]
for each closed two-sided superideal $J$ of $R$. Indeed, if $J=D^{\perp}$, then 
\[(R/J)^{\circ}=\varinjlim_{L}(C^*/(D^{\perp}+L^{\perp}))^*=\varinjlim_{L}(C^*/(D\cap L)^{\perp})^*\simeq \varinjlim_{L}((D\cap L)^*)^*=D,\]
where $L$ runs over all finite-dimensional subsupercoalgebras of $C$. This implies that $m$ is injective.

On the other hand, if $J$ is a closed superideal of $R$, then $J=\varprojlim_{I=L^{\perp}}(J+I)/I$, and the superideal $(J+I)/I$ of the finite-dimensional superalgebra $R/I\simeq L^*$ coincides with $D_L^{\perp}$ for the uniquely defined subsupercoalgebra $D_L\subseteq L$. Thus \[J=\varprojlim_{L} D_L^{\perp}=(\varinjlim_{L} D_L)^{\perp}=D^{\perp},\] where $D=\varinjlim_{L} D_L$ is a subsupercoalgebra of $C$, showing that $m$ is surjective.
\end{proof}
We remark that the supercoalgebra $(R/J)^{\circ}$ can be identified with the subsupercoalgebra 
\[\widehat{J}=\{c\in C\mid \hat{c}(j)=j(c)=0 \ \mbox{for every} \ j\in J\}\]
of $C$.

Let $C^*-SDis$ denote the category of discrete left $C^*$-supermodules (considered with graded morphisms).  
There is a natural functor $SComod^C\to C^*-SDis$ that identifies objects as superspaces. Moreover, if $M$ is a $C$-supercomodule, then
\[\phi\cdot m=\sum (-1)^{|\phi||m_1|}m_1\phi(c_2), \]
for $m\in M, \phi\in C^* $ and $\tau_M(m)=\sum m_1\otimes c_2$,
defines the structure of $C^*$-supermodule on $M$. The following lemma is a superization of Theorem 4.3(a) of \cite{sims}.  
\begin{lm}\label{super Simpson}
The functor $Scomod^C\to C^*-SDis$ is a duality.
\end{lm}

Let $R$ be a pseudocompact superalgebra. A right $R$-supermodule $M$ is said to be \emph{pseudocompact}, if $M$  is homeomorphic to the inverse limit of 
finite-dimensional discrete $R$-supermodules. The category of right pseudocompact $R$-supermodules with graded continuous morphisms is denoted by $SPC-R$. 
Define the functor $R-SDis\to SPC-R$ via $S\mapsto S^*$, where a basis of the pseudocompact topology of $S^*$ consists of its subsupersubmodules $L^{\perp}$ for all finite-dimensional $R$-subsupermodules $L$ of $S$. 

Further, define the functor $SPC-R\to R-SDis$ by $\ M\mapsto M^{\circ}$, where $M^{\circ}$ consists of all $\phi\in M^*$ such that $\ker\phi$ contains an open subsupermodule $N$ of $M$.

The following statement is an easy superization of Proposition 2.6(d) of  \cite{sims}.
\begin{lm}\label{duality between PC and Dis}
The above functors $R-SDis\to SPC-R$ and $SPC-R\to R-SDis$ are dualities that are quasi-inverse of each other.
\end{lm}

\section{Distribution superalgebra}

Let $G$ be a supergroup. Denote $\ker\epsilon_G=K[G]^+$ by $\mathfrak{m}$, and for any integer $t\geq 0$, denote  
$(K[G]/\mathfrak{m}^{t+1})^*$ by $Dist_t(G)$.
The \emph{distribution superalgebra} $Dist(G)$ of $G$ is defined as $Dist(G)=\cup_{t\geq 0} Dist_t(G)\subseteq K[G]^*$. 

The superspace $Dist(G)$ has the natural structure of a Hopf superalgebra (see \cite{brunkuj, zub2, zub3} for more details).  
The multiplication in $Dist(G)$ is induced by the multiplication in $K[G]^*$. To define a supercoalgebra structure on $Dist(G)$, use that for every $t\leq t'$, the superspace embedding $Dist_t(G)\to Dist_{t'}(G)$ is a morphism of supercoalgebras that is dual to the superalgebra morphism $K[G]/\mathfrak{m}^{t'+1}\to K[G]/\mathfrak{m}^{t+1}$.

The superspace $Dist_1(G)^+ =(\mathfrak{m}/\mathfrak{m}^2)^*$ is identified with the Lie superalgebra $\mathfrak{g}$ of $G$, and it coincides with the subsuperspace consisting of all primitive elements of $Dist(G)$.

Since $Dist(G)$ is a subsuperalgebra of $K[G]^*$, every $G$-supermodule $M$ is a $Dist(G)$-supermodule as well.
\begin{lm}\label{action on tensor product}(Lemma 11.1 of \cite{zub2})
If $M$ and $N$ are $G$-supermodules, then $Dist(G)$ acts on $M\otimes N$ by the rule
\[\phi\cdot (m\otimes n)=\sum (-1)^{|\phi_2||m|}\phi_1 m\otimes \phi_2 n,\]
where $\Delta_{Dist(G)}(\phi)=\sum \phi_1\otimes\phi_2$, $m\in M, n\in N$ and $\phi\in Dist(G)$. In particular, if $\phi$ is primitive, say $\phi\in \mathfrak{g}$, then
\[\phi\cdot (m\otimes n)=\phi\cdot m\otimes n+(-1)^{|\phi||m|}m\otimes\phi\cdot n.\]
\end{lm}

For example, let $G=GL(m|n)$. The elements $e_{ij}$
such that $e_{ij}(x_{kl}-\delta_{kl})=\delta_{ik}\delta_{jl}$ for $1\leq i, j, k, l\leq m+n$
generate the Lie superalgebra $\mathfrak{g}$. Moreover,
$Dist(G)=Dist(G)_{\mathbb{Z}}\otimes_{\mathbb{Z}} K$, where $Dist(G)_{\mathbb{Z}}$ is a Hopf subsuperring
of the \emph{universal enveloping} superalgebra $U(\mathfrak{g})$ of $\mathfrak{g}$, generated by the elements 
\[e_{ij}^{(t)}=\frac{e_{ij}^t}{t!} \text{ for }1\leq i\neq j\leq m+n \text{ and } t\geq 1,\]
and \[\binom{e_{ii}}{s}=\frac{e_{ii}(e_{ii}-1)\ldots (e_{ii}-s+1)}{s!} \text{ for } 1\leq i\leq m+n \text{ and } s\geq 0.\] 
Here, $t\leq 1$ whenever $|e_{ij}|=1$.
\begin{lm}\label{variation on lem3.1}
Let $M_1, \ldots, M_k$ be a collection of $G$-supermodules. If $e_{ij}$ is even, then $e_{ij}^{(t)}$ acts on $M_1\otimes\ldots\otimes M_k$ by the rule
\[e_{ij}^{(t)}\cdot(m_1\otimes\ldots\otimes m_k)=\sum_{t_1+\ldots+t_k=t} e_{ij}^{(t_1)}\cdot m_1\otimes\ldots\otimes e_{ij}^{(t_k)}\cdot m_k,\]
where $m_s\in M_s$ for $1\leq s\leq k$. If $e_{ij}$ is odd, then
\[e_{ij}\cdot (m_1\otimes\ldots\otimes m_k)=\sum_{1\leq s\leq k}(-1)^{\sum_{1\leq j< s}|m_j|}m_1\otimes\ldots\otimes e_{ij}\cdot m_s\otimes\ldots\otimes m_k.\]
\end{lm}
\begin{proof}
Use the formula
\[\Delta_{Dist(G)}(e_{ij}^{(t)})=\sum_{0\leq k\leq t}e_{ij}^{(k)}\otimes e_{ij}^{(t-k)},\]
Lemma \ref{action on tensor product}, and the induction on $k$.
\end{proof}

\section{Standard and costandard supermodules over general linear supergroup}

Here we overview the properties of standard and costandard objects in the highest weight category $G-Smod$.

In what follows, let $G$ denote the general linear supergroup $GL(m|n)$ unless stated otherwise. Let $B^-$, and $B^+$ respectively, denote the standard Borel subsupergroups of $G$, consisting of the lower and upper triangular matrices, respectively. Then $B^+\cap B^-$ is a maximal torus of $G$, denoted by $T$.

Let $H^0_{-}(\lambda^{\epsilon})$, and $H^0_{+}(\lambda^{\epsilon})$ respectively, denote the induced $G$-supermodules \linebreak $\mathrm{ind}^G_{B^{-}} K^{\epsilon}_{\lambda}$ and $\mathrm{ind}^G_{B^{+}} K^{\epsilon}_{\lambda}$, respectively, where $K^{\epsilon}_{\lambda}$ is regarded as an irreducible $B^-$- or $B^+$-supermodule of weight $\lambda\in X(T)$ of parity $\epsilon=0, 1$. 

It was shown in \cite{zub1} that $H^0_-(\lambda^{\epsilon})\neq 0$ if and only if $\lambda$ is a {\it dominant} weight, that is 
\[\lambda_1\geq\ldots\geq\lambda_m, \ \lambda_{m+1}\geq\ldots\geq\lambda_{m+n}.\]
Moreover, if $\lambda$ is dominant, then the socle of $H_-^0(\lambda^{\epsilon})$ is isomorphic to the irreducible supermodule $L_-(\lambda^{\epsilon})$, and any irreducible $G$-supermodule is isomorphic to some $L_-(\lambda^{\epsilon})$ (see also \cite{shib}).

Let $X(T)^+$ denote the set of all dominant weights. It is a \emph{poset} with respect to the {\it dominant or Bruhat} order, which will be denoted by $\unlhd$ (see \cite{brunkuj, zub1}). We extend this partial order for the set $X(T)^+\times\{0, 1\}$ in such a way that $\lambda^{\epsilon}\unlhd \mu^{\epsilon'}$ whenever $\lambda\unlhd\mu$.

Analogously, $H_+^0(\lambda^{\epsilon})\neq 0$ if and only if the weight $\lambda$ is {\it anti-dominant}, that is
\[\lambda_1\leq\ldots\leq\lambda_m, \ \lambda_{m+1}\leq\ldots\leq\lambda_{m+n}.\]
Also, the socle of $H_+^0(\lambda^{\epsilon})$ is isomorphic to the irreducible supermodule $L_+(\lambda^{\epsilon})$, 
and any irreducible $G$-supermodule is isomorphic to some $L_+(\lambda^{\epsilon})$.

Let $X(T)^-$ denote the set of all anti-dominant weights. As above, the sets $X(T)^-$ and $X(T)^-\times\{0, 1\}$ are posets with respect to the \emph{anti-dominant} order, which is just the opposite order to the dominance order $\unlhd$. 

The map $t: x_{ij}\mapsto (-1)^{|i|(|i|+|j|)}x_{ji}$ induces an anti-automorphism of the Hopf superalgebra $K[G]$. The corresponding self-duality of the full subcategory of all finite-dimensional $G$-supermodules is denoted by $M\mapsto M^{<t>}$.

Set $V_{\pm}(\lambda^{\epsilon})=H^0_{\pm}(\lambda^{\epsilon})^{<t>}$. 
Then $H^0_{-}(\lambda^{\epsilon})$ and $V_{-}(\lambda^{\epsilon})$ form complete collections of costandard and standard objects of the highest weight category $G-Smod$, whose irreducible objects are indexed by the elements of $X(T)^+\times\{0, 1\}$. 
Symmetrically, the supermodules $H^0_{+}(\lambda^{\epsilon})$ and $V_{+}(\lambda^{\epsilon})$ form complete collections of costandard and standard objects of the same category, and the irreducible objects are indexed by the elements of $X(T)^-\times\{0, 1\}$.

Moreover, it has been observed in \cite{sz} that
\[H^0_{-}(\lambda^{\epsilon})^*\simeq V_{+}(-\lambda^{\epsilon}), \ V_{-}(\lambda^{\epsilon})^*\simeq H^0_{+}(-\lambda^{\epsilon}).\]
This implies that $L_-(\lambda^{\epsilon})^*\simeq L_+(-\lambda^{\epsilon})$.

To simplify the notation, let $H^0_{\pm}(\lambda)$, $V_{\pm}(\lambda)$ and $L_{\pm}(\lambda)$, respectively, denote $H^0_{\pm}(\lambda^0)$, $V_{\pm}(\lambda^0)$ and $L_{\pm}(\lambda^0)$, respectively. Then $H^0_{\pm}(\lambda^{\epsilon})\simeq\Pi^{\epsilon}H_{\pm}^0(\lambda)$, $V_{\pm}(\lambda^{\epsilon})\simeq\Pi^{\epsilon}V_{\pm}(\lambda)$, and $L_{\pm}(\lambda^{\epsilon})\simeq\Pi^{\epsilon}L_{\pm}(\lambda)$.

Recall Definition 3.8 of \cite{markozub} stating that a weight $\mu$ is a \emph{predecessor} of $\lambda$ if
$\mu\lhd\lambda$, and there is no weight $\pi$ such that $\mu\lhd\pi\lhd\lambda$.
Since any weight $\lambda\in X(T)^+$ has only finitely many predecessors, the poset $X(T)^+$ is good in the sense of Definition 3.9 of \cite{markozub} (cf. \cite{sz}, Example 5.1). 
The posets $X(T)^+\times\{0, 1\}$, $X(T)^-$ and $X(T)^-\times \{0,1\}$ are also good.

Let $V$ be the natural $GL(m|n)$-supermodule of superdimension $m|n$ defined earlier.
Let $P^-=\mathrm{Stab}_G(V_1)$ and $P^+=\mathrm{Stab}_G(V_0)$ be the standard parabolic subsupergroups of $G$. 
Denote by $U^{-}$ and $U^+$, respectively, the kernels of the natural epimorphisms $P^{-}\to G_{ev}$ and $P^+\to G_{ev}$, respectively. Then  $P^{-}=U^{-}\rtimes G_{ev}$ and $P^{+}=U^{+}\rtimes G_{ev}$.
Since $U^-$ and $U^+$ are purely odd unipotent supergroups, they are finite and infinitesimal.

Analogously as above, we define the $G_{ev}$-modules $H^0_{ev, \pm}(\lambda)=\mathrm{ind}^{G_{ev}}_{B^{\pm}_{ev}} K_{\lambda}$ and $V_{ev, \pm}(\lambda)=H^0_{ev, \pm}(\lambda)^{<t>}$. The socle of $H^0_{ev, \pm}(\lambda)$ is an irreducible $G_{ev}$-module, which is isomorphic to the top of $V_{ev, \pm}(\lambda)$, and it is denoted by $L_{ev, \pm}(\lambda)$. There is
$L_{ev, -}(\lambda)^*\simeq L_{ev, +}(-\lambda)\simeq L_{ev, -}(-w_0\lambda)$, where $w_0$ is the longest element of the {\it even Weyl group} $W_0$ of $G$ (see \cite{jan}, Corollary II.2.5).

Any $G_{ev}$-supermodule $M$  can be regarded as a $P^-$-supermodule (and a $P^+$-supermodule, respectively) via the epimorphism $P^-\to G_{ev}$ (and  $P^+\to G_{ev}$, respectively).
Then, by Lemma 5.2 of \cite{zub1} and Lemma 8.5 of \cite{zub2}, we obtain
\[H^0_-(\lambda)\simeq\mathrm{ind}^G_{P^-} H^0_{ev, -}(\lambda), V_-(\lambda)\simeq\mathrm{coind}^G_{P^+} V_{ev, -}(\lambda)=Dist(G)\otimes_{Dist(P^+)} V_{ev, -}(\lambda).\]
Analogously, there are natural isomorphisms
\[H^0_+(\lambda)\simeq\mathrm{ind}^G_{P^+} H^0_{ev, +}(\lambda), V_+(\lambda)\simeq\mathrm{coind}^G_{P^-} V_{ev, +}(\lambda)=Dist(G)\otimes_{Dist(P^-)} V_{ev, +}(\lambda).\]

The supersubgroups $P^-$ and $P^+$ are $t$-connected. In particular, the correspondence $M\mapsto M^{<t>}$ induces a duality between the category of all finite-dimensional $P^-$-supermodules
and the category of all finite-dimensional $P^+$-supermodules.

\begin{lm}\label{dualitybetweenfunctors}
For every finite-dimensional $P^-$-supermodule $M$, there is a natural isomorphism of $G$-supermodules
\[(\mathrm{ind}^G_{P^-} M)^{<t>}\simeq\mathrm{coind}^G_{P^+} M^{<t>}.\]
\end{lm}
\begin{proof}
For a finite-dimensional $G$-supermodule $N$ there is a natural isomorphism
\[\mathrm{Hom}_G(\mathrm{coind}^G_{P^+} M^{<t>}, N)\simeq\mathrm{Hom}_{P^+}(M^{<t>}, N)\simeq
\mathrm{Hom}_{P^-}(N^{<t>}, M).\]
In other words, the functor $M\mapsto \mathrm{coind}^G_{P^+} M^{<t>}$ is left-adjoint to the functor 
$N\mapsto (N^{<t>})|_{P^-}$. Symmetrically, there is an isomorphism
\[\mathrm{Hom}_G((\mathrm{ind}^G_{P^-} M)^{<t>}, N)\simeq\mathrm{Hom}_G(N^{<t>}, \mathrm{ind}^G_{P^-} M)\simeq\mathrm{Hom}_{P^-}(N^{<t>}, M),\]
hence the functor $M\mapsto (\mathrm{ind}^G_{P^-} M)^{<t>}$ is left-adjoint to the same functor 
$N\mapsto (N^{<t>})|_{P^-}$. Therefore, the statement follows from Corollary IV.1 of \cite{mac}.  
\end{proof}

Following \cite{markozub}, an increasing filtration 
\[0=M_0\subseteq M_1\subseteq M_2\subseteq\ldots \]
of a $G$-supermodule $M$ such that $M=\cup_{i\geq 0}M_i$ and the quotient $M_i/M_{i-1}$ is isomorphic to a standard supermodule $V_-(\lambda_i^{\epsilon_i})$ for every  $i\geq 1$, is called a {\it standard or Weyl} filtration. Symmetrically, a decreasing filtration
\[N=N_0\supseteq N_1\supseteq N_2\supseteq\ldots\]
of a $G$-supermodule $N$, such that $\cap_{i\geq 0}N_i=0$ and the quotient $N_i/N_{i+1}$ is isomorphic to a costandard supermodule $H^0_-(\mu_i^{\tau_i})$ for every $i\geq 0$, is called a {\it costandard or good} filtration.

Let $\Gamma\subseteq X(T)^+$ be an ideal of weights. We  say that a supermodule $M\in G-SMod$ belongs to
$\Gamma$, if every irreducible composition factor of $M$ is of the form $L_-(\lambda^{\epsilon})$, where 
$\lambda\in\Gamma$. 
 
Finally, assume that the ideal $\Gamma$ is finitely generated. A $G$-supermodule $N$ is called $\Gamma$-{\it restricted}, if $N$ belongs to $\Gamma$ and $[N : L_-(\lambda^{\epsilon})]<\infty$ for every 
$\lambda^{\epsilon}\in\Gamma\times\{0, 1\}$. 

\section{Donkin-Koppinen filtration}

As above, $G$ denotes the general linear supergroup $GL(m|n)$. 

In what follows, $X(T)^-\times X(T)^+$ is regarded as a poset with the componentwise ordering.  
Then $(G\times G)-Smod$ is the highest weight category whose standard and costandard objects are 
$H^0_+(\lambda^{\epsilon})\otimes H^0_-(\mu^{\pi})$ and  $V_+(\lambda^{\epsilon})\otimes V_-(\mu^{\pi})$, 
respectively (see \cite{sz}).

Consider $K[G]$ as a (left) $G\times G$-supermodule via $\rho_l\times\rho_r$. Let $\Lambda$ be an ideal in $X(T)^-\times X(T)^+$. Denote by  $O_{\Lambda}(K[G])$ the largest $G\times G$-subsupermodule of $K[G]$ that belongs to $\Lambda$. If $\Lambda$ is finitely generated, then
$O_{\Lambda}(K[G])$ is $\Lambda$-restricted.  

\begin{tr}\label{Donkin-Koppinen}(Theorem 6.1 of \cite{sz})
For every finitely generated ideal $\Lambda\subseteq X(T)^-\times X(T)^+$, the $G\times G$-supermodule $O_{\Lambda}(K[G])$ has a decreasing good filtration
\[O_{\Lambda}(K[G])=V_0\supseteq V_1\supseteq V_2\supseteq\ldots\]
such that 
\[V_k/V_{k+1}\simeq H^0_+(-\lambda_k)\otimes H_-^0(\lambda_k)\simeq V_-(\lambda_k)^*\otimes H^0_-(\lambda_k)\]
for $k\geq 0$. 
\end{tr}

Any filtration of $O_{\Lambda}(K[G])$ as in the above theorem will be called a \emph{Donkin-Koppinen filtration.}
For more information about Donkin-Koppinen filtrations, consult \cite{kop}, 1.4 of \cite{don}, or Proposition 4.20 in part II of \cite{jan}.

From now on, we choose an ideal $\Lambda$ to be of the form $(-\Gamma)\times\Gamma$, where $\Gamma$ is a finitely generated ideal in $X(T)^+$ and $-\Gamma=\{-\lambda\mid \lambda\in\Gamma\}$.  
We choose a filtration 
\[\Gamma=\Gamma_0\supset\Gamma_1\supset\Gamma_2\supset \ldots \] 
such that each $\Gamma_k\setminus\Gamma_{k+1}$ consists of a single maximal element of $\Gamma_k$ for every $k\geq 0$. 
Then a filtration in Theorem \ref{Donkin-Koppinen} can be constructed by setting 
$V_k=O_{\Lambda_k}(K[G])$, where $\Lambda_k=(-\Gamma_k)\times\Gamma_k$ for $k\geq 0$. We call this filtration \emph{special}.

Consider $K[G]$ as a left $G$-supermodule with respect to $\rho_r$. Then $O_{\Gamma}(K[G])$ denotes the largest $G$-subsupermodule that belongs to $\Gamma$. The following proposition is a reformulation of Theorem 6.1 of \cite{sz}, and it also superizes 2.2(a) of \cite{don2}.  
\begin{pr}\label{bi-supercomodule=supercomodule}
There is $O_{\Lambda}(K[G])=O_{\Gamma}(K[G])$.
\end{pr}
\begin{proof}
It is clear that $O_{\Lambda}(K[G])\subseteq O_{\Gamma}(K[G])$. 

Let $L_-(\lambda^{\epsilon})$ be a composition quotient of $O_{\Gamma}(K[G])$. In other words, there is a finite-dimensional $G$-subsupermodule $M$ of $O_{\Gamma}(K[G])$, which has $L_-(\lambda^{\epsilon})$ as its composition factor. Let $N$ denote the $G\times G$-supermodule generated by $M$. Then $N$ is also finite-dimensional.

Since $K[G]=\varinjlim_{\Lambda'}O_{\Lambda'}(K[G])$, there is an ideal
$\Lambda'=(-\Gamma')\times\Gamma'$ such that $N\subseteq O_{\Lambda'}(K[G])$. Fix a special filtration 
\[O_{\Lambda'}(K[G])=V_0'\supseteq V'_1\supseteq V'_2\supseteq\ldots\]
of $O_{\Lambda'}(K[G])$, and denote by $k$ the minimal non-negative integer such that
$N\cap V_{k+1}'=0$. Without a loss of generality we can replace $O_{\Lambda'}(K[G])$ by a $G\times G$-supermodule $W=O_{\Lambda'}(K[G])/V'_{k+1}$ that has the finite filtration
\[W=W_0=V'_0/V'_{k+1}\supseteq\ldots\supseteq W_k=V'_k/V'_{k+1}\supseteq 0.\]
Using induction on $k$, we will show that $M$ is contained in a $G\times G$-subsupermodule $U$ of $W$ such that $U$ has a good filtration with quotients $H_+^0(-\mu)\otimes H_-^0(\mu)$ for $\mu\in\Gamma$.
This implies $M\subseteq N\subseteq O_{\Lambda}(K[G])$, and $O_{\Gamma}(K[G])\subseteq O_{\Lambda}(K[G])$.

Let $k=0$. Then $W_0$, with respect to the left regular action on $K[G]$, is isomorphic to 
a direct sum of $\dim V_-(\lambda_0')^*=\dim H_-^0(\lambda_0')$ copies of $H_-^0(\lambda_0')$.
In this case, the socle of $M$ contains an irreducible supermodule $L_-(\lambda_0')$, which implies 
$\lambda'_0\in\Gamma$.

If $k\geq 1$, then $(M+W_k)/W_k$ is contained in a $G\times G$-subsupermodule $U/W_k$ of $W_0/W_k$ that satisfies the induction hypothesis.
If $\lambda_k'\lhd\mu$ for some quotient $H_+^0(-\mu)\otimes H_-^0(\mu)$ of a good filtration of $U/W_k$,
then $\lambda_k'\in\Gamma$ and the statement for $M$ follows.
Otherwise, $\lambda_k'$ is incompatible with all weights $\mu$ such that $H_+^0(-\mu)\otimes H_-^0(\mu)$ is a quotient of a good filtration of $U/W_k$. In this case, Lemma 4.2 of \cite{markozub} implies that
the embedding $W_k\to U$ splits, that is, $U=R\oplus W_k$ and $R\simeq U/W_k$. Since the projection $U\to W_k$ maps $M$ to zero, we have $M\subseteq R$, which implies $N\subseteq R$. Thus $N\cap V'_k=0$, contradicting the minimality of $k$.  
\end{proof}

For a dominant weight $\lambda$, let $W(\lambda)$ denote a finite-dimensional $G$-supermodule that satisfies the following conditions:

\begin{enumerate}
\item $V_-(\lambda)$ is a subsupermodule of $W(\lambda)$; 
\item $H^0_-(\lambda)$ is a quotient of $W(\lambda)$;
\item all weights of $W(\lambda)$ are less or equal to $\lambda$;
\item $\dim W(\lambda)_{\lambda}=1$.
\end{enumerate}
The above conditions do not define $W(\lambda)$ uniquely. 
We can use the following proposition to construct various $W(\lambda)$ for a given weight $\lambda$. 

\begin{pr}\label{atensorproduct}
Let $M$ and $N$ be finite-dimensional tilting $G_{ev}$-modules such that all weights of $M$ and $N$ are less than or equal to $\mu$ and $\nu$, respectively. If $\dim M_{\mu}\neq 0$ and $\dim N_{\nu}\neq 0$, then the $G$-supermodule
\[W=(\mathrm{coind}^G_{P^+} M)\otimes (\mathrm{ind}^G_{P^-} N)\]
has a quotient isomorphic to $H^0_-(\mu+\nu)$ and a subsupermodule isomorphic to $V_-(\mu+\nu)$.
\end{pr}
\begin{proof}
Since $M\otimes N$ is a tilting $G_{ev}$-module of highest weight $\mu+\nu$, there is an $G_{ev}$-epimorphism
\[M\otimes N\simeq (\mathrm{coind}^G_{P^+} M)_{U^-}\otimes N\to H^0_{ev}(\mu+\nu).\]
Since the supergroup $U^-$ acts on both $N$ and $H_{ev}(\mu+\nu)$ trivially, the last morphism extends to 
the epimorphism of $P^-$-supermodules
\[(\mathrm{coind}^G_{P^+} M)\otimes N\to H^0_{ev}(\mu+\nu).\]  
Since $G/P^-$ is an affine superscheme, the functor $\mathrm{ind}^G_{P^-}$ is (faithfully) exact (cf. \cite{zub3}, Theorem 5.2).
The first statement now follows by application of the tensor identity.

Next, $V_-(\mu+\nu)$ is a subsupermodule of $W$ if and only if $H^0_-(\mu+\nu)$ is a quotient of 
$W^{<t>}\simeq
(\mathrm{ind}^G_{P^-} M^{<t>})\otimes (\mathrm{coind}^G_{P^{+}} N^{<t>})$. Since $M\simeq M^{<t>}$ and $N\simeq N^{<t>}$ as $G_{ev}$-modules, we can proceed analogously as in the proof of the firs statement.
\end{proof}

\begin{rem}
If $N$ is a tilting $G_{ev}$-module of highest weight $\lambda$ (cf. \cite{don1}), then the $G$-supermodule $V_-(0)\otimes\mathrm{ind}^G_{P^-} N$ has a quotient isomorphic to $H_-^0(\lambda)$ and a subsupermodule isomorphic to $V_-(\lambda)$. Therefore, we can take $W(\lambda)= V_-(0)\otimes\mathrm{ind}^G_{P^-} N$.
\end{rem}

\begin{pr}\label{likeDonkin}
Let $V_k$ and $\lambda_k$ be as in Theorem  \ref{Donkin-Koppinen}.
If $W_k=W(\lambda_k)$, then $\rho_{W_k}(W_k^*\otimes W_k)+V_{k+1}=V_k$.
\end{pr}
\begin{proof}
For the sake of simplicity, denote $\rho_{W_k}$ by $\rho_k$.
For a weight $\mu\in X(T)^+$ denote by $(\mu]$ the interval $\{\pi\in X(T)^+ \mid \pi\unlhd\mu\}$. From the definition of $W_k$, it follows that all weights of $G\times G$-supermodule $W_k^*\otimes W_k$ belong to $\Lambda_k=(-(\lambda_k])\times (\lambda_k]\subseteq\Lambda$. Therefore, $\rho_k(W_k^*\otimes W_k)\subseteq V_k$. 

Let $R$ denote $\ker (W_k\to H^0_-(\lambda_k))$ and $S$ denote $W_k/V_-(\lambda_k)$. Assume that 
\[\rho_k(S^*\otimes W_k+W_k^*\otimes R)\not\subseteq V_{k+1}.\]
Then $S^*\otimes W_k+W_k^*\otimes R$ contains a vector of weight $(-\lambda_k, \lambda_k)$, which is a pre-image of the primitive (highest weight) vector of $V_-(\lambda_k)^*\otimes H_-^0(\lambda_k)\simeq H^0_+(-\lambda_k)\otimes H^0(\lambda_k)$.
 On the other hand, there is a unique (up to a non-zero scalar multiple) vector in  $(W_k^*\otimes W_k)_{(-\lambda_k, \lambda_k)}$ that is given as $w^*\otimes w$, where $w$ is a generator of $(W_k)_{\lambda_k}$. However, $w^*\otimes w$ does not belong to $S^*\otimes W_k+W_k^*\otimes R$, which is a contradiction. 
Thus
\[\rho_k(S^*\otimes W_k+W_k^*\otimes R)\subseteq V_{k+1},\]
and $\rho_k$ induces a supermodule morphism
\[V_-(\lambda_k)^*\otimes H_-^0(\lambda_k)\simeq W_k^*\otimes W_k/(S^*\otimes W_k+W_k^*\otimes R)\to
V_k/V_{k+1}\simeq V_-(\lambda_k)^*\otimes H_-^0(\lambda_k).\]
Moreover, the induced morphism takes the primitive vector on the left to the primitive vector on the right; hence this is an embedding. The claim now follows by comparing dimensions.
\end{proof}

\begin{cor}\label{Donkin-Koppinenrealization}
There is $O_{\Lambda}(K[G])=\sum_{k\geq 0}\rho_k(W_k^*\otimes W_k)$.
\end{cor}
\begin{proof}
It is clear that $M=\sum_{k\geq 0}\rho_k(W_k^*\otimes W_k)\subseteq O_{\Lambda}(K[G])$. 
On the other hand, both $G\times G$-supermodules $M$ and $O_{\Lambda}(K[G])$ are restricted (see \cite{sz}, Lemma 6.1).  Furthermore, for any $(\mu, \pi)\in\Lambda$, the irreducible $G\times G$-supermodule $L_+(\mu)\otimes L_-(\pi)$, or its parity shift, appears as a composition factor in only finitely many quotients $V_k/V_{k+1}$.  
Thus there is a non-negative integer $t$ such that $[V_k/V_{k+1} : \Pi^{\epsilon}(L_+(\mu)\otimes L_-(\pi))]=0$ for all $k> t$, and $[O_{\Lambda}(K[G]) :\Pi^{\epsilon}( L_+(\mu)\otimes L_-(\pi))]=
[O_{\Lambda}(K[G])/V_{t+1} : \Pi^{\epsilon}(L_+(\mu)\otimes L_-(\pi))]$. Since the embedding $M\to O_{\Lambda}(K[G])$ induces an epimorphism $M\to O_{\Lambda}(K[G])/V_{t+1}$, one sees that $[M : \Pi^{\epsilon}(L_+(\mu)\otimes L_-(\pi))]\geq [O_{\Lambda}(K[G]) : \Pi^{\epsilon}(L_+(\mu)\otimes L_-(\pi))]$ for any $(\mu, \pi)\in\Lambda$. Therefore, $M=O_{\Lambda}(K[G])$. 
\end{proof}

\section{The explicit description of Donkin-Koppinen filtration}

Assume again that $G=GL(m|n)$. There is a commutative diagram of superschemes
\[\begin{array}{rcl}
 & G & \\
\swarrow & & \searrow \\
 U^-\times P^+ & & P^-\times U^+ \\
 \searrow & & \swarrow \\ 
 & U^-\times G_{ev}\times U^+
\end{array},\]
where the superscheme isomorphisms $u_1 : G\to U^-\times P^+$ and $u_2 : G\to P^-\times U^+$ are given as
\[\left(\begin{array}{cc}
A_{11} & A_{12} \\
A_{21} & A_{22}
\end{array}\right)\mapsto \left(\begin{array}{cc}
I_m & 0 \\
A_{21}A_{11}^{-1} & I_n
\end{array}\right)\times \left(\begin{array}{cc}
A_{11} & A_{12} \\
0 & A_{22}-A_{21}A_{11}^{-1}A_{12}
\end{array}\right)\] 
and 
\[\left(\begin{array}{cc}
A_{11} & A_{12} \\
A_{21} & A_{22}
\end{array}\right)\mapsto \left(\begin{array}{cc}
A_{11} & 0 \\
A_{21} & A_{22}-A_{21}A_{11}^{-1}
\end{array}\right)\times \left(\begin{array}{cc}
I_m & A_{11}^{-1}A_{12} \\
0 & I_n
\end{array}\right),\]
respectively.

Moreover, the superscheme isomorphisms $v_1 : U^-\times P^+\to U^-\times G_{ev}\times U^+$ and 
$v_2 : P^-\times U^+\to U^-\times G_{ev}\times U^+$ are identity maps on $U^-$ and $U^+$ respectively, and defined on  $P^+$ and $P^-$ as
\[\left(\begin{array}{cc}
P_{11} & P_{12} \\
0 & P_{22}
\end{array}\right)\mapsto \left(\begin{array}{cc}
P_{11} & 0 \\
0 & P_{22}
\end{array}\right)\times\left(\begin{array}{cc}
I_m & P_{11}^{-1}P_{12} \\
0 & I_n
\end{array}\right)\]
and 
\[\left(\begin{array}{cc}
P_{11} & 0 \\
P_{21} & P_{22}
\end{array}\right)\mapsto \left(\begin{array}{cc}
I_m & 0 \\
P_{21}P_{11}^{-1} & I_n
\end{array}\right)\times \left(\begin{array}{cc}
P_{11} & 0 \\
0 & P_{22}
\end{array}\right)\]
correspondingly.

Denote the isomorphism $v_1u_1=v_2u_2$ by $\phi$. Then 
\[\left(\begin{array}{cc}
A_{11} & A_{12} \\
A_{21} & A_{22}
\end{array}\right)\stackrel{\phi}{\mapsto} \]
\[\left(\begin{array}{cc}
I_m & 0 \\
A_{21}A_{11}^{-1} & I_m
\end{array}\right)\times \left(\begin{array}{cc}
A_{11} & 0 \\
0 & A_{22}-A_{21}A_{11}^{-1}A_{12}
\end{array}\right)\times \left(\begin{array}{cc}
I_m &A_{11}^{-1} A_{12} \\
0 & I_n
\end{array}\right).\]
The dual isomorphism of coordinate superalgebras 
\[\phi^* : K[U^-]\otimes K[G_{ev}]\otimes K[U^+]\to K[G]\]
is defined by
\[Y_{21}\mapsto X_{21}X_{11}^{-1}, \quad Y_{11}\mapsto X_{11}, \quad Y_{22}\mapsto X_{22}-X_{21}X_{11}^{-1}X_{12}, \quad \text{and} \quad 
Y_{12}\mapsto X_{11}^{-1}X_{12}.\]

Let $R$ be an algebraic supergroup and $H$ its subsupergroup. Suppose that there are an affine superscheme $U$  and a superscheme isomorphism $\alpha : R\to H\times U$ that commutes with the natural right action of $H$ on both $G$ and $H\times U$. 
Let $W$ be a right $K[H]$-supercomodule, i.e., a left $H$-supermodule. By Lemma 8.2 of \cite{zub4},
there is a  superspace isomorphism $W\otimes K[U]\to\mathrm{ind}^R_H W$ given by
\[w\otimes u\mapsto \sum w_1\otimes \alpha^*(h_2\otimes u), \]
where $w\otimes u\in W\otimes K[U]$ and $\tau_W(w)=\sum w_1\otimes h_2$.
\begin{lm}\label{first map}
Let $N$ be a subsupercomodule of $K[H]$ (regarded as a right supercomodule over itself). Then $\alpha^*(N\otimes K[U])$ is a right subsupercomodule of $K[R]$ that isomorphic to $\mathrm{ind}^R_H N$.
\end{lm} 
\begin{proof}
There is a commutative diagram 
\[\begin{array}{ccccc}
K[H]\otimes K[U] & \to & \mathrm{ind}^R_{H} K[H] & \to & K[R] \\
\uparrow & & \uparrow & & \\
N\otimes K[U] & \to &  \mathrm{ind}^R_{H} N & & 
\end{array},\]
where the rightmost horizontal and vertical arrows are morphisms of $R$-super-modules.
The isomorphism $ind^R_H K[H]\to K[R]$ is given by 
\[h\otimes f\mapsto \epsilon_H(h)f\] for $h\otimes f\in K[H]\otimes K[R]$.
Therefore, the composition of morphisms in the upper horizontal line equals $\alpha^*$. 
\end{proof}
For $1\leq j\leq m < i\leq m+n$, denote $(Y_{21}Y_{11}^{-1})_{ij}$ by $z_{ij}$. That is, 
\[z_{ij}=\sum_{1\leq k\leq m}y_{ik}y_{kj}^{(-1)}, \text{ where } y_{kj}^{(-1)}=(Y_{11}^{-1})_{kj}.\]

Denote by $L$ the subspace of $K[G_{ev}]$ generated by all non-constant monomials in the "variables" $y_{uv}y^{(-1)}_{st}$ for $1\leq s, t\leq m< u, v\leq m+n$.
 It is easy to see that $L$ is a right subcomodule of $K[G_{ev}]$, and also a right (purely even) subsupercomodule of $K[P^-]$. 
\begin{lm}\label{invariants modulo T}
Every expression $z_{ij}$ is invariant modulo $L$.
\end{lm}
\begin{proof}
We have
\[\Delta_{P^-}(z_{ij})=\sum_{1\leq k\leq m}\sum_{1\leq s\leq m+n, 1\leq s'\leq m} y_{is}y_{s' j}^{(-1)}\otimes y_{sk}y_{k s'}^{(-1)}=
\]
\[\sum_{1\leq s, s'\leq m}y_{is}y_{s' j}^{(-1)}\otimes\sum_{1\leq k\leq m}y_{sk}y_{k s'}^{(-1)}+
\sum_{1\leq s'\leq m < s\leq m+n}\sum_{1\leq k\leq m}y_{is}y_{s' j}^{(-1)}\otimes y_{sk}y_{k s'}^{(-1)}=\]
\[z_{ij}\otimes 1 \ \mathrm{mod} (L\otimes K[P^-]).\]
\end{proof}

Every $\lambda\in X(T)$ can be expressed as an ordered pair $(\lambda_+\mid\lambda_-)$, where $\lambda_+=(\lambda_1, \ldots, \lambda_m)$ is a weight of $GL(m)$ and $\lambda_-=(\lambda_{m+1}, \ldots, \lambda_{m+n})$ is a weight of $GL(n)$. The \emph{strong} Bruhat-Tits order $\unlhd_s$ on $X(T)$ is defined by 
$\lambda\unlhd_s\mu$ if and only if $\lambda_+\unlhd_{GL(m)} \mu_+$ and $\lambda_-\unlhd_{GL(n)} \mu_-$ with respect the Bruhat-Tits (or dominant) orders $\unlhd_{GL(m)}$ and $\unlhd_{GL(n)}$ on the weight lattices of $GL(m)$ and $GL(n)$, respectively. Clearly, $\lambda\unlhd_s\mu$ implies $\lambda\unlhd\mu$, but the converse is not true. In particular, if $\Gamma$ is an ideal in $X(T)^+$, then it is also an ideal with respect to the strong Bruhat-Tits order 
$\unlhd_s$.

For $\Gamma\subseteq X(T)^+$, denote $M_{\Gamma}=\mathcal{O}_{\Gamma}(K[G_{ev}])$
and $N_{\Gamma}=K[Y_{21}Y_{11}^{-1}]M_{\Gamma}$. 

The module $M_{\Gamma}$ is a direct sum of some members of a Donkin-Koppinen filtration of $K[G_{ev}]$.
Indeed, assume that the ideal $\Gamma$ is generated by weights $\lambda^{(1)}, \ldots, \lambda^{(s)}$. 
Without a loss of generality, we can assume that they are pairwise incomparable.
For $1\leq i\leq s$, denote $r^{(i)}_+ =|\lambda^{(i)}_+|, r^{(i)}_- =|\lambda^{(i)}_-|$. 
An ordered pair of integers $(a, b)$ is called \emph{admissible}, if there is a nonnegative integer $l$ and an index $i$ such that $a=r^{(i)}_+-l, b=r^{(i)}_- +l$.
Further, denote by $\Gamma_{a, b}$ the set $\{\mu\in\Gamma\mid |\mu_+|=a, |\mu_-|=b\}$. 
It is easy to see that
\[\Gamma_{a, b}=\cup_{i, l=r^{(i)}_+ -a=b- r^{(i)}_-}\{\mu\mid \mu\unlhd_s\lambda^{(i)}-l(\epsilon_m-\epsilon_{m+1})\}\]
and 
\[\Gamma=\sqcup_{(a, b) \ \mbox{is admissible}}\Gamma_{a, b}.\]
We call this decomposition of $\Gamma$ \emph{admissible}.

Since each subset $\Gamma_{a, b}$ is a finite ideal with respect to the order $\unlhd_s$, the subcoalgebra $M_{\Gamma}$ is a direct sum of finite-dimensional subcoalgebras as 
\[M_{\Gamma}=\oplus_{(a, b) \ \mbox{is admissible}} M_{\Gamma_{a, b}}.\]
Moreover, every maximal element $\lambda\in\Gamma$ belongs to some $\Gamma_{a, b}$. Therefore,
\[M_{\Gamma}/M_{\Gamma\setminus\{\lambda\}}\simeq M_{\Gamma_{a, b}}/M_{\Gamma_{a, b}\setminus\{\lambda\}}
\simeq H^0_{ev, -}(\lambda^*)\otimes H_{ev, -}^0(\lambda)\simeq V_{ev, -}(\lambda)^*\otimes H_{ev, -}(\lambda)\] (see (2.2a) of \cite{don2}).
\begin{lm}\label{a filtration for P^-}
If  $\Gamma\subseteq X(T)^+$ is  a finitely generated ideal, then 
$N_{\Gamma}\subseteq K[P^-]$ is a subsupercomodule of $K[P^-]$. Moreover, for every maximal element $\lambda$ of $\Gamma$, the quotient $N_{\Gamma}/N_{\Gamma\setminus\lambda}$ is isomorphic to the direct sum  
\[H_{ev, -}^0(\lambda)^{\oplus \frac{\dim H^0_-(\lambda)}{2}}\oplus \Pi H_{ev, -}^0(\lambda)^{\oplus \frac{\dim H^0_-(\lambda)}{2}},\]
regarded as a left $P^-$-supermodule.
\end{lm}
\begin{proof}
The weight of each monomial generator of $L$ is a sum (with repetitions) of the weights $-(\epsilon_i-\epsilon_j)$ for 
$1\leq i\leq m < j\leq m+n$. Thus, 
$LM_{\Gamma}\subseteq M_{\Gamma\setminus\pi}$ for every ideal $\Gamma$ and every maximal element $\pi$ of 
$\Gamma$. The first statement now follows from Lemma \ref{invariants modulo T}.

As it has been observed, $M_{\Gamma}/M_{\Gamma\setminus\lambda}$ is isomorphic to a direct sum of $\dim H_{ev, -}^0(\lambda)$ copies of $H^0_{ev, -}(\lambda)$. Since $K[Y_{21}Y_{11}^{-1}]$ is the Grassman algebra on $Y_{21}Y_{11}^{-1}$, there is the equality of dimensions $\dim K[Y_{21}Y_{11}^{-1}]_0=\dim K[Y_{21}Y_{11}^{-1}]_1$.
The inclusion
$LM_{\Gamma}\subseteq M_{\Gamma\setminus\lambda}$, combined with Lemma \ref{invariants modulo T}, implies
that $N_{\Gamma}/N_{\Gamma\setminus\lambda}$ is isomorphic to a direct sum of 
\[\frac{\dim K[Y_{21}Y_{11}^{-1}]}{2}\dim H^0_{ev, -}(\lambda)=\frac{\dim H^0_-(\lambda)}{2}\] copies of $H^0_{ev, -}(\lambda)$ and the same number of copies of $\Pi H^0_{ev, -}(\lambda)$ (both regarded as $P^-$-supermodules).
\end{proof}

\begin{theorem}\label{DK?}
For every finitely generated ideal $\Gamma\subseteq X(T)^+$, the superspace 
\[C_{\Gamma}=\phi^*(K[Y_{21}]\otimes M_{\Gamma}\otimes K[Y_{12}])\] is a left $G$-subsupermodule of
$K[G]$ with respect to $\rho_r$. Moreover, for any maximal element $\lambda$ of $\Gamma$,  the quotient $C_{\Gamma}/C_{\Gamma\setminus\lambda}$ is isomorphic to the direct sum  
\[H_-^0(\lambda)^{\oplus \frac{\dim H^0_-(\lambda)}{2}}\oplus \Pi H_-^0(\lambda)^{\oplus \frac{\dim H^0_-(\lambda)}{2}}.\]
Therefore, $C_{\Gamma}=O_{\Gamma}(K[G])$.
\end{theorem}
\begin{proof}
Since 
\[\phi^*(K[Y_{21}]\otimes M_{\Gamma}\otimes K[Y_{12}])=u_2^*(N_{\Gamma}\otimes K[Y_{12}]),\]
for $u_2 : G\to P^-\times U^+$ as before, 
Lemma \ref{first map} implies the first claim. 

The second claim follows by the exactness of the functor $\mathrm{ind}^G_{P^-}$ and  the second statement of Lemma \ref{a filtration for P^-}.  

Finally, $C_{\Gamma}\subseteq O_{\Gamma}(K[G])$, and the good filtrations of both
$C_{\Gamma}$ and $O_{\Gamma}(K[G])$, regarded as left $G$-supermodules with respect to $\rho_r$, are identical to each other.
Thus $C_{\Gamma}=O_{\Gamma}(K[G])$.
\end{proof}
Because $\phi^*$ is a superspace isomorphism, we infer that $C_{\Gamma}$ is a direct sum of subsuperspaces
\[C_{\Gamma}=\oplus_{(a, b) \ \mbox{is admissible}} C_{\Gamma_{a, b}},\]
where $C_{\Gamma_{a, b}}=\phi^*(K[Y_{21}]\otimes M_{\Gamma_{a, b}}\otimes K[Y_{12}])$.

Each $C_{\Gamma_{a, b}}$ is generated by elements $\phi^*(u\otimes v\otimes w)$, where
$u, w$ are elements of bases of $K[Y_{21}]$ and $K[Y_{12}]$, respectively, and $v$ is an element of a basis of
$M_{\Gamma_{a, b}}$. Therefore, we obtain the following corollary.
\begin{cor}\label{the basis!}
A union of all elements $\phi^*(u\otimes v\otimes w)$ as above, for all admissible $(a,b)$, is a basis of $C_{\Gamma}$.
Moreover, for each maximal element $\lambda$ of $\Gamma$, the quotient  
$C_{\Gamma}/C_{\Gamma\setminus\{\lambda\}}$ has a basis $\phi^*(u\otimes v\otimes w)$, where $u,w$ are as above, and $v$ runs over a basis of $M_{\Gamma}/M_{\Gamma\setminus\{\lambda\}}$.
\end{cor}

\section{A basis of $C_{\Gamma}$ }

In this section, we use the combinatorics of tableaux and bideterminants to construct a basis of $C_{\Gamma}$ explicitly. The reader is asked to consult \cite{martin} for more details and explanations.

Throughout this section, $\Gamma$ is a finitely generated ideal of $X(T)^+$, and $\lambda$ is a maximal element of $\Gamma$. Set
$a=\min\{\lambda_m, 0\}$ and $b=\min\{\lambda_{m+n}, 0\}$, respectively. Let $\mu$ denote the dominant weight $\lambda-a\epsilon_m-b\epsilon_{m+n}$. The even and odd part $\mu_+$ and $\mu_-$ of the weight $\mu$ are partitions.

Denote by $\nu=\mu'$ the weight $(\mu_+'\mid\mu_-')$, where $\pi'$ denotes the partition \emph{conjugated} to a partition $\pi$.

Consider a $G_{ev}=GL(m)\times GL(n)$-module 
\[T=T_1\otimes\ldots\otimes T_{\mu_1}\otimes T_{\mu_1+1}\otimes\ldots T_{\mu_1+\mu_{m+1}}\otimes T_{\mu_1+\mu_{m+1}+1}\otimes T_{\mu_1+\mu_{m+1}+2},\]
where 
$T_i=\Lambda^{\nu_i}(V_0)$ for $1\leq i\leq \mu_1$,  $T_{\mu_1+j}=\Lambda^{\nu_{\mu_1+j}}(V_1)$ for $1\leq j\leq \mu_{m+1}$,
and
$T_{\mu_1+\mu_{m+1}+1}=\Lambda^m(V_0^*)^{\otimes |a|}$ and $T_{\mu_1+\mu_{m+1}+2}=\Lambda^n(V_1^*)^{\otimes |b|}$.
Let $s$ denote the integer $\mu_1+\mu_{m+1}+2$.

By Lemma (3.4) of \cite{don1}, (see also E.6 (2) of  \cite{jan}), $T$ is a tilting $G_{ev}$-module of highest weight $\lambda$. In particular, $T$ has a good filtration with $H^0_{ev, -}(\lambda)$ at the top and a Weyl filtration with $V_{ev, -}(\lambda)$ at the bottom. 

Arguing as in Proposition \ref{likeDonkin} and Corollary \ref{Donkin-Koppinenrealization}, one sees that $\rho_T(T^*\otimes T)$ is congruent to $M_{\Gamma}$ modulo $M_{\Gamma\setminus\{\lambda\}}$. 
Moreover, since the matrix 
\[Y_{ev}=\left(\begin{array}{cc}
Y_{11} & 0 \\
0 & Y_{22}
\end{array}\right)\]
represents the element $\mathrm{id}_{K[G_{ev}]}$ in 
\[G_{ev}(K[G_{ev}])\simeq GL(m)(K[G_{ev}])\times GL(n)(K[G_{ev}]),\] Lemma \ref{another formula} implies that $\rho_T(\phi)=\mathrm{tr}(Y_{ev}\circ\phi)$ for every $\phi\in T^*\otimes T\simeq \mathrm{End}_K(T)$.

The space $\mathrm{End}_K(T)$ is spanned by the \emph{decomposable} endomorphisms of the type 
$\phi_1\otimes\ldots\otimes\phi_s$, where $\phi_i\in\mathrm{End}_K(T_i)$ for each $1\leq i\leq s$. 
Since $\rho_T(\phi)$ is linear in $\phi$, $\rho_T(End_K(T))$ is generated by the elements
\[\mathrm{tr}(Y_{ev}\circ(\phi_1\otimes\ldots\otimes\phi_s))=\mathrm{tr}((Y_{ev}\circ\phi_1)\otimes\ldots\otimes (Y_{ev}\circ\phi_s))=\mathrm{tr}(Y_{ev}\circ\phi_1)\ldots \mathrm{tr}(Y_{ev}\circ\phi_s).\]

Let $W$ be a vector space of dimension $l=\dim W<\infty$. Let $w_1, \ldots, w_l$ be a basis of $W$ and 
$w_1^*, \ldots, w_l^*$ be the dual basis of $W^*$. The space $\mathrm{End}_K(\Lambda^k(W))\simeq
(\Lambda^k(W))^*\otimes\Lambda^k(W)\simeq \Lambda^k(W^*)\otimes \Lambda^k(W)$ is generated by the elements
\[\phi_{ij}=w^*_{i_1}\wedge\ldots\wedge w^*_{i_k}\otimes w_{j_1}\wedge\ldots\wedge w_{j_k},\]
where $1\leq i_1 < \ldots < i_k\leq l, \ 1\leq j_1 < \ldots < j_k\leq l$.
Denote by $Y$ a generic matrix with the entry $y_{uv}$ at the position $(u,v)$ corresponding to $1\leq u, v\leq l$.   
\begin{lm}\label{known formula}
There is $\mathrm{tr}(Y\circ\phi_{ij})=\det((y_{i_u j_v})_{1\leq u, v\leq k})$. 
\end{lm}
\begin{proof}
Straightforward calculations. 
\end{proof}
Combining all the above remarks, one can easily derive that $\rho_T(\mathrm{End}_K(T))$ is spanned by the products of various minor determinants of $Y_{11}$ of orders $\nu_1, \ldots, \nu_{\mu_1}$, minor determinants of $Y_{22}$ of orders $\nu_{\mu_1+1}, \ldots, \nu_{\mu_1+\mu_{m+1}}$, and $\det(Y_{11})^a$ and \linebreak $\det(Y_{22})^b$. In the notations from Definition 2.4.1 of \cite{martin}, every such product is a \emph{generalized bideterminant}
\[T^{\mu_+}(i^+ : j^+)T^{\mu_-}(i^- : j^-)\det(Y_{11})^a \det(Y_{22})^b ,\]
where $i^+, j^+\in I(m, r^+), i^-, j^-\in I(n, r^-), r^+=|\mu_+|, r^-=|\mu_-|$.
Extending \cite{martin}, we call $T^{\mu_+}(i^+ : j^+)T^{\mu_-}(i^- : j^-)$ a \emph{bideterminant of shape} $\mu$. The above expression $T^{\mu_+}(i^+ : j^+)T^{\mu_-}(i^- : j^-)\det(Y_{11})^a \det(Y_{22})^b$ is called a \emph{generalized bideterminant of shape} $\lambda$.

Finally, a bideterminant $T^{\mu_+}(i^+ : j^+)T^{\mu_-}(i^- : j^-)$ and a generalized biderminant $T^{\mu_+}(i^+ : j^+)T^{\mu_-}(i^- : j^-)\det(Y_{11})^a \det(Y_{22})^b$ are called \emph{standard}, if the tableaux
$T^{\mu^+}_{i^+}, T^{\mu^+}_{j^+}, T^{\mu^-}_{i^-}$ and $T^{\mu^-}_{j^-}$ are standard.

\begin{lm}\label{basis of even DK quotient}
The standard generalized bideterminants of shape $\lambda$ form a basis of
$M_{\Gamma}/M_{\Gamma\setminus\{\lambda\}}$.
\end{lm}
\begin{proof}
There is an isomorphism $H^0_{ev, -}(\lambda)\simeq H^0_{ev, -}(\mu)\otimes T_{s-1}\otimes T_s$. By Proposition 4.20, in part II of  \cite{jan}, there is
\[\dim M_{\Gamma}/M_{\Gamma\setminus\{\lambda\}}=(\dim H^0_{ev, -}(\lambda))^2=(\dim H^0_{ev, -}(\mu))^2.\]
On the other hand, by Proposition 2.5.5 of \cite{martin}, every bideterminant of shape $\mu$
is the sum of a linear combination of standard bideterminants of shape $\mu$, and a linear combination of standard bideterminants of shape $\pi$ such that either $\pi_+$ is strictly less than $\mu_+$ or $\pi_-$ is strictly less than $\mu_-$, with respect to the \emph{reverse lexicographical} order (see  Definition 2.5.4 of\cite{martin}). In both cases, $\pi\lhd_s\mu$, hence $\pi\lhd\mu$. Proposition \ref{likeDonkin} implies that $M_{\Gamma}/M_{\Gamma\setminus\{\lambda\}}$ is spanned by the standard generalized bideterminants of shape $\lambda$. Theorem 3.2.6 of \cite{martin}, concludes the proof.
\end{proof}

The following theorem follows immediately.
\begin{theorem}\label{basis of a DK factor}
Using the notation established earlier, every factor $C_{\Gamma}/C_{\Gamma\setminus\{\lambda\}}$ has a $K$-basis consisting of elements
\[\prod_{1\leq i\neq j \leq m+n, |y_{ij}|=1}\phi^*(y_{ij})^{\epsilon_{ij}} \phi^*(T^{\mu_{+}}(i^{+} : j^{+})T^{\mu_{-}}(i^{-} : j^{-})\det(Y_{11})^a \det(Y_{22})^b),\]
where the tableaux
$T^{\mu^+}_{i^+}, T^{\mu^+}_{j^+}$, $T^{\mu^-}_{i^-}$ and $T^{\mu^-}_{j^-}$ are standard.
Consequently, as in Corollary \ref{the basis!}, we obtain an infinite basis of $C_{\Gamma}$ by combining the above basis  elements for factors of a filtration 
\[\Gamma=\Gamma_0\supset \Gamma_1 \supset \ldots \]
where each $\Gamma_k\setminus \Gamma_{k+1}$ consists of a single maximal element of $\Gamma_k$ for every $k\geq 0$.
\end{theorem}

\section{Generalized Schur superalgebras}\label{8}
From now on, we denote the category  $G-Smod$ by $\mathcal{C}$. If $\Gamma$ is an ideal of $X(T)^+$, then $\mathcal{C}[\Gamma]$ denotes the full subcategory of $\mathcal{C}$ consisting of all supermodules belonging to $\Gamma$. According to Proposition 3.7 of \cite{markozub},  $\mathcal{C}[\Gamma]$ is the highest weight category.

If $\Gamma$ is a finitely generated ideal of $X(T)^+$ and $\Lambda=(-\Gamma)\times\Gamma$,  
then $C_{\Gamma}=O_{\Gamma}(K[G])=O_{\Lambda}(K[G])$ is a subsupercoalgebra of $K[G]$. The anti-homomorphism of supergroups $G\to G\times G$, given by $g\mapsto (g^{-1}, 1),$ induces the structure of a right $G$-supermodule on $O_{\Lambda}(K[G])$, that is, $O_{\Lambda}(K[G])$ is a left $K[G]$-supercomodule with respect to $\Delta_G$. Thus
\[\Delta_G(C_{\Gamma})\subseteq K[G]\otimes C_{\Gamma}\cap C_{\Gamma}\otimes K[G]=C_{\Gamma}\otimes C_{\Gamma}.\] 

Let $\Gamma$ be an arbitrary ideal of $X(T)^+$. 
Then $C_{\Gamma}=\varinjlim_{\Gamma'\subseteq\Gamma} C_{\Gamma'}$, where $\Gamma'$ runs over all finitely generated subideals of $\Gamma$. Therefore, $C_{\Gamma}$ is also a subsupercoalgebra of $K[G]$. 

Denote by $S_{\Gamma}$ the associative pseudocompact superalgebra $C_{\Gamma}^*$. 
The superalgebra $S_{\Gamma}$ is called a \emph{generalized Schur superalgebra}. 

\begin{lm}\label{subcategories and generalized Schur superalgebras}
For an ideal $\Gamma$ of $X(T)^+$ there is an equality $\mathcal{C}[\Gamma]=SComod^{C_{\Gamma}}\simeq S_{\Gamma}-SDis$.
\end{lm}
\begin{proof}
Since objects of both categories $\mathcal{C}[\Gamma]$ and $SComod^{C_{\Gamma}}$ are unions of 
finite (i.e., finite-dimensional) subobjects, and any finite object in these categories belongs to a finitely generated ideal
$\Gamma'\subseteq\Gamma$, one can assume that $\Gamma$ is finitely generated. 

If $V\in \mathcal{C}[\Gamma]$, then Lemma \ref{canonicalmap} shows that $\mathrm{cf}(V)\subseteq C_{\Gamma}$.
This implies $\mathcal{C}[\Gamma]\subseteq SComod^{C_{\Gamma}}$.
Conversely, any finite-dimensional $C_{\Gamma}$-supercomodule is embedded in a direct sum of finitely many copies of $C_{\Gamma}$ or its parity shift. Thus Proposition \ref{bi-supercomodule=supercomodule} implies 
$SComod^{C_{\Gamma}}\subseteq\mathcal{C}[\Gamma]$.
\end{proof}

Let $X(T)^+_{l}$ denote the ideal $\{\lambda\in X(T)^+\mid |\lambda|=l\}$, where $l$ is an integer. Each set  $X(T)^+_{l}$ is also a coideal of $X(T)^+$.

We have $X(T)^+ =\sqcup_{l\in\mathbb{Z}} X(T)^+_{l}$ since for $l\neq s$ the weights $\lambda\in X(T)^+_{l}$ and $\mu\in X(T)^+_s$ are not comparable.

If $\Gamma$ is an ideal of $X(T)^+$, then $\Gamma=\sqcup_{l\in\mathbb{Z}}\Gamma_{l}$, where $\Gamma_{l}=\Gamma\cap X(T)^+_{l}$. The ideal $\Gamma$ is finitely generated if and only if each $\Gamma_{l}$ is finitely generated, and all but finitely many $\Gamma_{l}$ are empty. 

Consequently,  $C_{\Gamma}=\oplus_{l\in\mathbb{Z}} C_{\Gamma_{l}}$, which implies that $S_{\Gamma}$ is isomorphic to $\prod_{l\in\mathbb{Z}} S_{\Gamma_{l}}$ (equipped with the product topology). It also implies the following lemma. 
\begin{lm}\label{category decomposition}
The category $\mathcal{C}[\Gamma]$ decomposes as $\mathcal{C}[\Gamma]=\oplus_{l\in\mathbb{Z}} \mathcal{C}[\Gamma_{l}]$.
In other words, each object $M\in\mathcal{C}[\Gamma]$ decomposes as $M=\oplus_{l\in\mathbb{Z}} M_{l}$, where $M_{l}\in\mathcal{C}[\Gamma_{l}]$ for each $l$.
\end{lm}
\begin{proof}
Let $e_l$ denote the unit element of $S_{\Gamma_l}$. Then $1=\prod_{l\in\mathbb{Z}} e_l$ and every $M\in\mathcal{C}[\Gamma]$ decomposes as $M=\oplus_{l\in\mathbb{Z}} e_l M$. In fact, for each $m\in M$, all but finitely many $e_{l}$ vanish on $m$, that is $m\in \oplus_{l\in\mathbb{Z}} e_l M$.
\end{proof}
In what follows, let $S_{l}$ denote $S_{X(T)^+_{l}}$. Each $S_{\Gamma_{l}}$ is a factor of $S_{l}$.  

\begin{lm}\label{a set of generators}
The superalgebra $K[G]$ is generated by the elements $x_{ij}$ and $s_G(x_{ij})=x_{ij}^{(-1)}$ for $1\leq i, j\leq m+n$.
\end{lm}
\begin{proof}
Let $A$ denote the subsuperalgebra of $K[G]$ generated by the above elements. It is clear that $A$ is a Hopf subsuperalgebra of $K[G]$. Let $H$ be an algebraic supergroup such that $K[H]\simeq A$. The natural embedding $A\to K[G]$ is dual to the epimorphism $\pi : G\to H$. 
Denote by $R$ the kernel of $\pi$. Then $K[R]\simeq K[G]/I_R$, where the Hopf superideal $I_R$ is generated by 
$A^+=A\cap\ker(\epsilon_G)$ (cf. Proposition 5.2 of \cite{zub3}). Since $A^+$ contains all elements $x_{ij}-\delta_{ij}$, we infer that $I_R=K[G]^+$ and $K[R]=K$. Therefore, $R=1$.
\end{proof}

Each monomial
\[x=\prod_{1\leq i, j\leq m+n} x_{ij}^{k_{ij}} \prod_{1\leq i, j\leq m+n} (x_{ij}^{(-1)})^{s_{ij}}\] 
has weight 
\[\lambda(x)=\sum_{1\leq j\leq m+n}(\sum_{1\leq i\leq m+n}k_{ij}-\sum_{1\leq i\leq m+n}s_{ij})\epsilon_j\]
with respect to $\rho_r$.

For an ordered pair of nonnegative integers $l, s$,  let $C_{l, s}$ denote the subsuperspace of $K[G]$ spanned by all monomials $x$ as above, such that $|\lambda(x)_+|=l$ and $|\lambda(x)_-|=s$. 

Analogously to \cite{dippdoty}, we derive the following.
\begin{lm}\label{monster Schur superalgebra}
Every $C_{l, s}$ is a finite-dimensional subsupercoalgebra of $K[G]$. Moreover, for every integer $l$, the subsuperspace $C_{l}=\cup_{k\geq\min\{-l, 0\}}C_{l+k, k}$ is a subsupercoalgebra of $K[G]$ that coincides with $C_{X(T)^+_{l}}$. In particular, $S_{l}\simeq C_{l}^*$.
\end{lm}
\begin{proof}
The first statement is obvious.
If $k\geq\min\{-l, 0\}$, then $C_{l+k, k}\subseteq C_{l+k+1, k+1}$ (see  (2.4.1) of \cite{dippdoty}), showing that $C_{l}$ is a subsupercoalgebra of $K[G]$. Lemma \ref{a set of generators} implies $K[G]=\oplus_{l\in\mathbb{Z}} C_{l}$. Since $C_{l}\subseteq C_{X(T)^+_{l}}$ and $K[G]=\oplus_{l\in\mathbb{Z}}C_{X(T)^+_{l}}$, the claim follows.
\end{proof}

\section{$S_{\Gamma}$ as an ascending quasi-hereditary superalgebra}

Superizing Definition 3.11 from \cite{markozub}
we call a pseudocompact superalgebra $A$ \emph{ascending quasi-hereditary} whenever $A$ has an ascending chain of closed superideals
\[0=H_0\varsubsetneq H_1\varsubsetneq H_1\varsubsetneq\ldots\varsubsetneq H_n\varsubsetneq\ldots \]
such that \\
(1) for every open right superideal $I$ of $A$ there is an index $t$ such that $H_t\not\subseteq I$.\\
Additionally, we require that for every $n\geq 1$, the following conditions hold: \\
(2) $H_n/H_{n-1}$ is a projective pseudocompact $A/H_{n-1}$-supermodule with finitely many indecomposable projective factors. \\
(3) $\mathrm{Hom}_{SPC-A}(H_n/H_{n-1}, A/H_n)=0$. \\
(4) $\mathrm{Hom}_{SPC-A}(H_n/H_{n-1}, \mathrm{rad}(H_n/H_{n-1}))=0$.

Recall from \cite{markozub} that if $A$ is a superalgebra, then its \emph{bozonization} $\widehat{A}$ is the semi-direct product algebra $A\rtimes\mathbb{Z}_2$. Arguing as in Lemma 7.6 of  \cite{markozub}, one can show that $SPC-A$ is equivalent to the category of pseudocompact $\widehat{A}$-modules $PC-\widehat{A}$. Thus, $A$ is an ascending quasi-hereditary superalgebra if and only if $\widehat{A}$ is an ascending quasi-hereditary algebra. By Theorem 3.12 and Theorem 3.15 of \cite{markozub}, $A-SDis$ is the highest weight category with respect to a good finitely generated poset if and only if $A$ is an ascending quasi-hereditary superalgebra.
\begin{lm}\label{radical}
Let $S$ be a discrete finite-dimensional $A$-supermodule. Then $\mathrm{rad}(S^*)$$=(\mathrm{socle}(S))^{\perp}$
$\simeq (S/\mathrm{socle}(S))^*$.
\end{lm}
\begin{proof}
Since $\mathrm{Ann}_A(S)$ is an open superideal of $A$, without loss a of generality, one can assume that $A$ is finite-dimensional. We replace $A$ by $\widehat{A}$, and $A-Smod$ by the equivalent category $\widehat{A}-mod$,  and use the functor $S\mapsto S^*$ which is a duality between full subcategories of finite-dimensional modules in $\widehat{A}-mod$ and $mod-\widehat{A}$ to derive the claim.
\end{proof}
Assume that $\Gamma$ is a finitely generated ideal of $X(T)^+$. Let
\[\Gamma=\Gamma_0\supset\Gamma_1\supset\Gamma_2\supset \ldots \]
be a special filtration of $\Gamma$. For each $n\geq 0$, denote by $H_n=C_{\Gamma_n}^{\perp}$ a two-sided superideal of $S_{\Gamma}$.
\begin{pr}\label{quasi-hereditariness}
The superalgebra $S_{\Gamma}$ is an ascending quasi-hereditary superalgebra with respect to the ascending chain of closed superideals \[0=H_0\subsetneq H_1\subsetneq\ldots\varsubsetneq H_n\varsubsetneq\ldots.\]
\end{pr}
\begin{proof}
Let $D\neq 0$ be a finite-dimensional subsupercoalgebra of $C_{\Gamma}$, and $I=D^{\perp}$ be an open superideal of $S_{\Gamma}$. Then $D\cap C_{\Gamma_n}=0$ for sufficiently large $n$, and Lemma \ref{correspondence} implies $H_n\not\subseteq I$.

By Lemma \ref{duality between PC and Dis}, for every $n\geq 1$, the right $S_{\Gamma_{n-1}}\simeq S_{\Gamma}/H_{n-1}$-supermodule $H_n/H_{n-1}\simeq (C_{\Gamma_{n-1}}/C_{\Gamma_n})^*$ is projective in $SPC-S_{\Gamma_{n-1}}$ if and only if the left $S_{\Gamma_{n-1}}$-supermodule $D_{n-1}=C_{\Gamma_{n-1}}/C_{\Gamma_n}$ is injective in $S_{\Gamma_{n-1}}-SDis\simeq\mathcal{C}[\Gamma_{n-1}]$. On the other hand, $D_{n-1}$ is a direct sum of finitely many copies of $H_-^0(\lambda_{n-1})$. Since $\lambda_{n-1}$ is maximal in $\Gamma_{n-1}$, $H_-^0(\lambda_{n-1})$ is the injective envelope of $L_-(\lambda_{n-1})$ in $\mathcal{C}[\Gamma_{n-1}]$ (see \cite{markozub}). 

There is     
\[\begin{aligned}\mathrm{Hom}_{SPC-S_{\Gamma}}(H_n/H_{n-1}, S_{\Gamma}/H_n)&\simeq\mathrm{Hom}_{S_{\Gamma}-SDis}(C_{\Gamma_n}, D_{n-1})\\
&\simeq\mathrm{Hom}_{\mathcal{C}[\Gamma_{n-1}]}(C_{\Gamma_n}, D_{n-1})=0,\end{aligned}\]
since the socle of $D_{n-1}$ does not belong to $\Gamma_n$. 
Furthermore, using Lemma \ref{radical} we infer
\[\begin{aligned}&\mathrm{Hom}_{SPC-S_{\Gamma}}(H_n/H_{n-1}, \mathrm{rad}(H_n/H_{n-1}))\\
&\simeq\mathrm{Hom}_{S_{\Gamma}-SDis}(D_{n-1}/\mathrm{socle}(D_{n-1}), D_{n-1})=0,\end{aligned}\]
since $L_-(\lambda_{n-1})$ does not appear as a composition factor of $D_{n-1}/\mathrm{socle}(D_{n-1})$.
We have verified conditions (1) through (4) from the definition of ascending quasi-hereditary superalgebra, proving the claim.
\end{proof}

\section{An approach to a description of $S_{\Gamma}$}

In the introduction of the paper, we have described an approach to a description of $S_{\Gamma}$.
Denote by $\pi_{\Gamma}$ the natural superalgebra morphism $\pi_{\Gamma} : Dist(G)\to S_{\Gamma}$ defined by $\xi\mapsto\xi|_{C_{\Gamma}}$ for $\xi\in Dist(G)$.

In this section, we first describe the induced topology on $Dist(G)$ corresponding to the morphism $\pi_{\Gamma}$.
Then we show that $S_{\Gamma}$ is completion of image $\pi_{\Gamma}(Dist(G))$ in the pseudocompact topology of $S_{\Gamma}$. We finish with preliminary results about the kernel of $\pi_{\Gamma}$.

It is obvious that $C_{l, s}=\mathrm{cf}(V^{\otimes l}\otimes (V^*)^{\otimes s})$, where $V$ is the natural $G$-supermodule of the superdimension $m|n$. Moreover,  $C_{l}=\cup_{k\geq\min\{-l, 0\}}\mathrm{cf}(V^{\otimes (l+k)}\otimes (V^*)^{\otimes k})$. 

For $l,s\geq 0$, set $S_{l, s}=C_{l, s}^*$. Each $S_{l, s}$ is a finite-dimensional superalgebra, which is called a \emph{standard rational Schur} superalgebra (compare with Definition 3.1 of \cite{dippdoty}). For example, $S_{l, 0}$ is the classical Schur superalgebra $S(m|n, l)$. Also, $S_{l}$ is obtained as $S_{l}=\varprojlim_{k\geq\min\{-l, 0\}} S_{l+k, k}$. 

Let $\pi_{l, s} : Dist(G)\to S_{l, s}$ be a superalgebra morphism as above, i.e., $\xi\mapsto\xi|_{C_{l, s}}$.  The following lemma describes the induced topology on $Dist(G)$.
\begin{lm}\label{induced topology}
A two-sided superideal $I$ of $Dist(G)$ is open if and only if $I$ contains an intersection of finitely many superideals $\ker\pi_{l, s}$.
\end{lm}
\begin{proof}
A two-sided superideal $I$ of $Dist(G)$ is open if and only if it has a form $Dist(G)\cap D^{\perp}$, where $D$ is a finite-dimensional subsupercoalgebra of $K[G]$. Lemma \ref{a set of generators} implies that $D\subseteq C_{l_1, s_1}+\ldots + C_{l_k, s_k}$ for some non-negative integers $l_1, s_1, \ldots, l_k,$ $s_k$. Then $\ker\pi_{l_1, s_1}\cap\ldots\cap\ker\pi_{l_k, s_k}\subseteq D^{\perp}$.
\end{proof}

\begin{lm}\label{density of Dist}
The image $\pi_{\Gamma}(Dist(G))$ is dense in $S_{\Gamma}$.
\end{lm}
\begin{proof}
If $D$ is a finite-dimensional subsupercoalgebra of $C_{\Gamma}$, then there is a nonnegative integer $l$ such that $D\cap \mathfrak{m}^{l+1}=0$, where $\mathfrak{m}=\ker\epsilon_G$. Therefore, every element of $D^*\simeq S_{\Gamma}/D^{\perp}$ can be lifted to an element of $Dist_l(G)$, which means $\pi_{\Gamma}(Dist(G))+D^{\perp}=S_{\Gamma}$.
\end{proof}

Denote $\pi_{X(T)^+_l}$ by $\pi_l$. The following result is now apparent.
\begin{cor}
An element $\xi\in Dist(G)$ belongs to $\ker\pi_{l}$ if and only if $\xi$ acts trivially on 
each $V^{\otimes (l+k)}\otimes (V^*)^{\otimes k}$.
\end{cor} 

\subsection{$ker(\pi_{\Gamma})$}

Let $M$ be a $G$-supermodule from Lemma \ref{canonicalmap}.
\begin{lm}\label{a criteria for vanishing}
An element $\xi\in Dist(G)$ vanishes on $\mathrm{Im}\rho_M$ if and only if $\xi$ acts trivially on $M$.
\end{lm} 
\begin{proof}Indeed, for every $\alpha\otimes m\in M^*\otimes M$ there is
\[\xi(\rho_M(\alpha\otimes m))=\sum\alpha(m_1)\xi(f_2)=\alpha(\sum m_1\xi(f_2))=(-1)^{|\xi|(|m|-|\xi|)}\alpha(\xi\cdot m).\]
Therefore, considering various homogeneous elements $\alpha$, one sees that $\xi(\mathrm{Im}(\rho_M))=0$ if and only if $\xi\cdot m=0$ for every homogeneous $m\in M$. 
\end{proof}

\begin{lm}\label{ker of a morphism}
An element $\xi\in Dist(G)$ belongs to $\ker\pi_{\Gamma}$ if and only if $\xi$ acts trivially on $W(\lambda)$ for every 
$\lambda\in\Gamma$. 
\end{lm}
\begin{proof}
Since $S_{\Gamma}=\varprojlim_{\Gamma'\subseteq\Gamma} S_{\Gamma'}$, where $\Gamma'$ runs over all finitely generated subideals of $\Gamma$, we infer that $\ker\pi_{\Gamma}=\cap_{\Gamma'\subseteq\Gamma}\ker\pi_{\Gamma'}$. 
Therefore, one can assume that $\Gamma$ is finitely generated. Corollary \ref{Donkin-Koppinenrealization} concludes the proof. 
\end{proof}

Recall that the map $t : x_{ij}\mapsto (-1)^{|i|(|i|+|j|)}x_{ji}$ induces an anti-automorphism of $K[G]$.
This anti-automorphism of $K[G]$ induces a superalgebra anti-automorphism of $Dist(G)$ defined as
$\xi\mapsto\xi^{<t>}$, where $\xi^{<t>}(f)=\xi(t(f))$ for $f\in K[G]$.
\begin{lm}\label{the kernel is invariant wrt transpose}
An element $\xi$ belongs to $\ker\pi_{\Gamma}$ if and only if $\xi^{<t>}$ does.
\end{lm} 
\begin{proof}
By Lemma 7.1 of \cite{zub2}, if $M$ is a finite-dimensional $G$-supermodule, then $M^{<t>}$ is identified with $M^*$ as a superspace on which $Dist(G)$ acts by the rule  
\[(\xi\phi)(m)=(-1)^{|\xi||\phi|}\phi(\xi^{<t>}m),\]
for $m\in M, \phi\in M^*$ and $\xi\in Dist(G)$.
Thus an element $\xi$ vanishes on $M^{<t>}$ if and only if $\xi^{<t>}$ vanishes on $M$. Since the map $M\mapsto M^{<t>}$ is a self-duality of the subcategory of $\mathcal{C}[\Gamma]$ consisting of all finite-dimensional supermodules, 
the claim follows.
\end{proof} 

\section{Generators of $\ker\pi_{\Gamma}\cap Dist(T)$}

Assume that $char K=p>0$. For the sake of simplicity, we write $e_i$ 
in place of $e_{ii}$ for $1\leq i\leq m+n$.

The distribution algebra $Dist (T)$ can be written as $Dist(T)=\varinjlim_{r\geq 0}  Dist(T_r)$, where $T_r$ is the $r$-th Frobenius kernel of $T$.
For an ideal $\Gamma\subseteq X(T)^+$, there is
\[\ker\pi_{\Gamma}\cap Dist(T)=\varinjlim_{r\geq 0}  Dist(T_r)\cap\ker\pi_{\Gamma},\] and thus, it suffices to describe each $Dist(T_r)\cap\ker\pi_{\Gamma}$.

Set $q=p^r$. For $0 \leq t \leq q-1$, denote 
\[h^{(q)}_t(x)=\sum_{t\leq k\leq q-1} (-1)^{k-t} \binom{k}{t}\binom{x}{k}.\] 
By Proposition 1.4 of \cite{markozub2}, each $Dist(T_r)$ is a separable commutative algebra, generated by
the pairwise orthogonal primitive idempotents 
\[h^{(q)}_{\alpha}(e)=\prod_{1\leq i\leq m+n}h^{(q)}_{\alpha_i}(e_i),\]
where $\alpha\in X(T)$ satisfies $0\leq\alpha_i\leq q-1$ for every $1\leq i\leq m+n$.

Let $W$ be a $G$-supermodule and $w\in W_{\mu}$ be a nonzero element of weight $\mu\in X(T)$. Then
\[h^{(q)}_{\alpha}(e)w=h^{(q)}_{\alpha}(\mu)w=(\prod_{1\leq i\leq m+n}h^{(q)}_{\alpha_i}(\mu_i))w.\]
The following lemma is now apparent.
\begin{lm}\label{the kernel in Dist(T)}
The ideal $Dist(T_r)\cap\ker\pi_{\Gamma}$ is generated (as a $K$-space) by all $h^{(q)}_{\alpha}(e)$ such that
$h^{(q)}_{\alpha}(\mu)=0$ for every weight $\mu$ appearing in some $W(\lambda)$ for $\lambda\in\Gamma$.
\end{lm}

\begin{lm}\label{lm3.1}
Assume $0\leq t <q$. If $m\equiv t \pmod q$, then $h_t^{(q)}(m)=1$, otherwise $h_t^{(q)}(m)=0$.
\end{lm}
\begin{proof}
The claim is trivial if $0 \leq m\leq t$.

In the remaining cases, we apply the identity $\binom{k}{t}\binom{m}{k}=\binom{m}{t}\binom{m-t}{k-t}$. 

If $t<m<q$, then we set $s=k-t$, and rewrite
\[h_t^{(q)}(m)=\binom{m}{t}\sum_{s=0}^{m-t} (-1)^s \binom{m-t}{s}=0.\]

If $m\geq q$, then we set $s=k-t$ again, and rewrite 
\[h_t^{(q)}(m)=\binom{m}{t}\sum_{s=0}^{q-1-t} (-1)^s \binom{m-t}{s}
=(-1)^{q-1-t}\binom{m}{t}\binom{m-1-t}{q-1-t}.\]
If $m\not\equiv t \pmod q$, then 
\[h_t^{(q)}(m)=(-1)^{q-1-t}\binom{m}{t}\binom{m-1-t}{q-1-t}=\frac{q}{m-t}\binom{m}{q}\binom{q-1}{t}\equiv 0 \pmod p.\]
If $m\equiv t \pmod q$, then
\[h_t^{(q)}(m)=(-1)^{q-1-t}\binom{m}{t}\binom{m-1-t}{q-1-t}\equiv (-1)^{q-1-t}\times 1\times (-1)^{q-1-t}\equiv 1 \pmod p.\]

For $m>0$ we have 
\[\begin{aligned}&h_t^{(q)}(-m)=\binom{-m}{t}\sum_{s=0}^{q-1-t}(-1)^s\binom{-m-t}{s}=\\
&(-1)^{t}\binom{m+t-1}{t}\sum_{s=0}^{q-1-t}\binom{m+t-1+s}{s}.\end{aligned}
\]

Applying the formula \[\sum_{s=0}^k \binom{l+s}{s}=\binom{l+k+1}{k}\] we get
\[\sum_{s=0}^{q-1-t}\binom{m+t-1+s}{s}=\binom{q+m-1}{q-1-t},\] 
and 
\[h_t^{(q)}(-m)= (-1)^t\binom{m+\alpha-1}{\alpha}\binom{q+m-1}{q-1-\alpha}.\]
If $m\not\equiv -t \pmod q$, then 
\[h_t^{(q)}(-m)= (-1)^t\binom{m+t-1}{t}\binom{q+m-1}{q-1-t}
=\]\[\frac{q}{m+t}\binom{q+m-1}{q}\binom{q-1}{t}\equiv 0 \pmod p.\]
If $m\equiv -t \pmod q$, then
\[h_t^{(q)}(-m)= (-1)^t\binom{m+t-1}{t}\binom{q+m-1}{q-1-t}\equiv (-1)^t\times (-1)^{t}\times 1 \equiv 1 \pmod p.\]
\end{proof}
\begin{cor}\label{if congruent}
There is  $h^{(q)}_{\alpha}(\mu)\neq 0$ if and only if 
$\alpha_i\equiv\mu_i \pmod q$ for each $1\leq i\leq m+n$.
\end{cor}

\begin{lm}\label{lm3.2}
Let $\lambda$ be a dominant weight of $GL(m|n)$.
Let $\alpha=(\alpha_1, \ldots, \alpha_{m+n})$, where $0\leq \alpha_i<q$ for each $i=1, \ldots, m+n$. 
Assume that $|\alpha|\equiv |\lambda| \pmod q$.
Then there is a dominant weight $\mu\unlhd \lambda$ such that $\mu_i\equiv \alpha_i \pmod q$ for each $0\leq i \leq m+n$.
\end{lm}
\begin{proof}
We construct a weight $\mu$ as follows. Choose $\mu_1$ to be the largest integer not exceeding $\lambda_1$ that is congruent to $\alpha_1$ modulo $q$. Then we choose $\mu_2$ to be the largest integer not exceeding $\mu_1$ and $\lambda_2$, that is congruent to $\alpha_2$ modulo $q$, and so on until $\mu_{m-1}$ to be the largest integer not exceeding $\mu_{m-2}$ and $\lambda_{m-1}$, that is congruent to $\alpha_{m-1}$ modulo $q$. If $\tilde{\mu}_m$ is the largest integer not exceeding $\mu_{m-1}$ and $\lambda_m$, such that $\tilde{\mu}_m\equiv \alpha_m \pmod q$, then we select $\mu_m=\tilde{\mu}_m - qt$ for sufficiently large $t>0$ to be determined later. By construction, it is clear that $\mu^+=(\mu_1, \ldots, \mu_m)$ is a dominant $GL(m)$-weight and $\mu^+\lhd \lambda^+$.

Next, we choose $\mu_{m+n}$ to be the smallest integer not smaller than $\lambda_{m+n}$ that is congruent to 
$\alpha_{m+n}$ modulo $q$. Then we choose $\mu_{m+n-1}$ to be the smallest integer not smaller than $\lambda_{m+n-1}$ and $\mu_{m+n}$, that is congruent to $\beta_{m+n-1}$ modulo $q$, and so on until $\mu_{m+2}$ to be the smallest integer not smaller than $\lambda_{m+2}$ and $\mu_{m+3}$, that is congruent to $\alpha_{m+2}$ modulo $q$.
We set $\mu_{m+1}=|\lambda|-|\mu^+|-\sum_{j=2}^n \mu_{m+j}$ so that $|\mu|=|\lambda|$.
The assumption $|\alpha|\equiv |\lambda| \pmod q$ implies that $\mu_{m+1}\equiv \alpha_{m+1} \pmod q$. If we choose $t$ large enough, we obtain that $\mu_{m+1}\geq \mu_{m+2}$. By construction, $\mu^-$ is a dominant weight 
of $GL(n)$. Finally, since $\mu^+\lhd \lambda^+$, $|\mu|=|\lambda|$, and $\mu_{m+j}\geq \lambda_{m+j}$ for each $1\leq j\leq n$, we infer that $\mu\lhd \lambda$.
\end{proof}

Let $\Gamma$ be a (not necessary finitely generated) ideal of weights. Let $\mathbb{Z}_{\Gamma}$ denote the subset $\{l\mid \Gamma_{l}\neq\emptyset\}$ of $\mathbb{Z}$.
\begin{pr}\label{lm3.3}
The space $ker(\pi_{\Gamma})\cap Dist(T)$ is the $K$-span of 
$h_{\alpha}^{(q)}(e)$ (over all $q$) such that $|\alpha|\not\equiv l \pmod q$ for every $l\in\mathbb{Z}_{\Gamma}$.
\end{pr}
\begin{proof}
Let $h^{(q)}_{\alpha}(e)$ be such that $|\alpha|\not\equiv l \pmod q$ for every $l\in\mathbb{Z}_{\Gamma}$. Assume there is a weight $\mu$ 
such that $W(\lambda)_{\mu}\neq 0$ for some $\lambda\in\Gamma$ such that $h^{(q)}_{\alpha}(\mu)\neq 0$.
Corollary \ref{if congruent} implies that $|\alpha|\equiv |\mu|=l\pmod q$, where $l\in \mathbb{Z}_{\Gamma}$. This contradiction shows that $h^{(q)}_{\alpha}(e)$ acts trivially on each $W(\lambda)$, where $\lambda\in \Gamma$. 
Lemma \ref{ker of a morphism} implies that $h^{(q)}_{\alpha}(e)\in ker(\pi_{\Gamma})\cap Dist(T_r)$.

Consider an element $h=\sum_{\alpha \text{ such that }|\alpha|\equiv l \pmod q \text{ for } l\in\mathbb{Z}_{\Gamma}} k_{\alpha} h_{\alpha}^{(q)}(e)$.
Fix  $\alpha$ such that $|\alpha|\equiv |\lambda|\pmod q$ for $\lambda\in\Gamma$, and  take $\mu$ as in Lemma \ref{lm3.2}, i.e., $\mu\unlhd\lambda$ and $\mu_i\equiv\alpha_i\pmod q$ for $1\leq i\leq m+n$.
Then Lemma \ref{lm3.1} implies $h(\mu)=k_{\alpha}$.
Thus, if $h$ is in $ker(\pi_{\Gamma})\cap Dist(T_r)$, then each coefficient $k_{\alpha}=0$, hence $h=0$.
\end{proof}

\section{Commutation formulae}

In this section, we derive certain commutation formulae that will be needed later. These formulae are also of independent interest.

Denote by $\overline{u}$ the representative of the congruence class of $u$ modulo $q$ such that $0\leq \overline{u}<q$.

Let $b=b_0+b_1p+\ldots+b_{r-1}p^{r-1}$ and $a=a_0+a_1p+\ldots+a_{r-1}p^{r-1}$ be $p$-adic expansions of 
$b$ and $a$. Then 
\[\binom{b}{a}\equiv \binom{b_0}{a_0}\binom{b_1}{a_1}\ldots \binom{b_{r-1}}{a_{r-1}}\pmod p.\]
For each $j=0, \ldots, r-1$ we denote $\binom{b_j}{a_j}$ by $\binom{b}{a}_j$.

\begin{pr} \label{commutation formulae}
For $1\leq t<q$ there are the following commutation relations:
\begin{enumerate}
\item $h^{(q)}_a(e_s)e_{ij}^{(t)}=e_{ij}^{(t)}h^{(q)}_a(e_s)$ if  $s\neq i,j$;
\item $h^{(q)}_a(e_i)e_{ij}^{(t)}=e_{ij}^{(t)}h^{(q)}_{\overline{a-t}}(e_i)$;
\item $h^{(q)}_a(e_j)e_{ij}^{(t)}=e_{ij}^{(t)} h^{(q)}_{\overline{a+t}}(e_j)$.
\end{enumerate}
\end{pr}
Using Lemma 7.7 of \cite{zubmarko}, we obtain the formulae $(1)$ immediately.

The proof of the remaining formulae will be given in the series of lemmas. 

Using Lemma 7.7 of \cite{zubmarko}, formulae (1) and (2) of \cite{m} and substitution $u=k-b$ we compute
\[\begin{aligned}h^{(q)}_a(e_i)e_{ij}^{(t)}=&e_{ij}^{(t)}\sum_{a\leq k\leq q-1} (-1)^{k-a}\binom{k}{a}\binom{e_i+t}{k}\\
=&e_{ij}^{(t)}\sum_{a\leq k\leq q-1} (-1)^{k-a}\binom{k}{a}\sum_{\min\{0,k-t\}\leq b\leq k}\binom{t}{k-b}\binom{e_i}{b}\\
=&e_{ij}^{(t)}\sum_{\min\{0,a-t\}\leq b\leq q-1}\Big[\sum_{b\leq k\leq \min\{b+t,q-1\}} (-1)^{k-a}\binom{k}{a}
\binom{t}{k-b}\Big]\binom{e_i}{b}\\
=&e_{ij}^{(t)}\sum_{\min\{0,a-t\}\leq b\leq q-1}\Big[\sum_{0\leq u\leq \min\{t, q-1-b\}}(-1)^{u+b-a}\binom{u+b}{a}\binom{t}{u}\Big]\binom{e_i}{b}
\end{aligned}\]
and
\[\begin{aligned}h^{(q)}_a(e_j)e_{ij}^{(t)}=&e_{ij}^{(t)}\sum_{a\leq k\leq q-1} (-1)^{k-a}\binom{k}{a}\binom{e_j-t}{k}\\
=&e_{ij}^{(t)}\sum_{a\leq k\leq q-1} (-1)^{k-a}\binom{k}{a}\sum_{0\leq b\leq k}(-1)^{k-b}\binom{t-1+k-b}{k-b}\binom{e_j}{b}\\
=&e_{ij}^{(t)}\sum_{0\leq b\leq q-1}\Big[\sum_{\max\{a,b\}\leq k\leq q-1} (-1)^{a+b}\binom{k}{a}
\binom{t-1+k-b}{k-b}\Big]\binom{e_j}{b}\\
=&e_{ij}^{(t)}\sum_{0\leq b\leq q-1}(-1)^{a+b}\Big[\sum_{\max\{a-b,0\}\leq u\leq q-1-b} \binom{u+b}{a}
\binom{t-1+u}{u}\Big]\binom{e_j}{b}.
\end{aligned}\] 

\begin{lm}
The formula (2) is valid for $t=1$.
\end{lm}
\begin{proof}
If $a>0$, then the sum \[\sum_{0\leq u\leq \min\{t, q-1-b\}}(-1)^{u+b-a}\binom{u+b}{a}\binom{t}{u}\]
equals \[(-1)^{b-a}\binom{b}{a}+(-1)^{1+b-a}\binom{b+1}{a}=(-1)^{b-a+1}\binom{b}{a-1}\] when $q-1-b\geq 1$, 
and equals \[(-1)^{q-1-a}\binom{q-1}{a}=(-1)^{q-a}\binom{q-1}{a-1}\] for $q-1-b=0$.  Thus
\[\sum_{b\leq k\leq \min\{b+t,q-1\}} (-1)^{k-a}\binom{k}{a}\binom{t}{k-b}=(-1)^{b-a+1}\binom{b}{a-1}\]
and
\[h^{(q)}_a(e_i)e_{ij}=e_{ij}\sum_{a-1\leq b\leq q-1}(-1)^{b-a+1}\binom{b}{a-1}\binom{e_{ii}}{b}=e_{ij}h^{(q)}_{a-1}(e_i).\]

If $a=0$, then the sum \[\sum_{0\leq u\leq \min\{t, q-1-b\}}(-1)^{u+b-a}\binom{u+b}{a}\binom{t}{u}\] vanishes when $q-1-b\geq 1$, and equals $1$ for $q-1-b=0$. Therefore
\[h^{(q)}_0(e_i)e_{ij}=e_{ij}\binom{e_i}{q-1}=e_{ij}h^{(q)}_{q-1}(e_i).\] 
\end{proof}

\begin{lm}
The formula (3) is valid for $t=1$.
\end{lm}
\begin{proof}
If $a<q-1$, then 
\[\sum_{\max\{a-b,0\}\leq u\leq q-1-b} \binom{u+b}{a}\binom{t-1+u}{u}=
\sum_{\max\{a-b,0\}\leq u\leq q-1-b} \binom{u+b}{a}.\]
If $a<b$, we use the formula 
\[\sum_{s=0}^k \binom{s+l}{l}=\binom{l+k+1}{k}\]
to rewrite that sum as 
\[\sum_{s=0}^{q-1-a}\binom{a+s}{a}-\sum_{s=0}^{b-a-1} \binom{a+s}{a}=\binom{q}{q-1-a}-\binom{b}{b-a-1}
=-\binom{b}{a+1}.\]
If $a\geq b$, then that sum equals \[\sum_{a-b\leq u\leq q-1-b} \binom{u+b}{a}=\sum_{s=0}^{q-1-a}\binom{a+s}{a}=\binom{q}{q-1-a}=0=-\binom{b}{a+1}.\]
Therefore, for $a<q-1$, we obtain 
\[h^{(q)}_a(e_j)e_{ij}=e_{ij}\sum_{0\leq b\leq q-1} (-1)^{a+b+1}\binom{b}{a+1}\binom{e_j}{b}=\]
\[e_{ij}\sum_{a+1\leq b\leq q-1} (-1)^{b-a-1}\binom{b}{a+1}\binom{e_j}{b}=e_{ij}h^{(q)}_{a+1}(e_j).\]
If $a=q-1$, then the sum \[\sum_{a-b\leq u\leq q-1-b} \binom{u+b}{a}=\binom{q-1}{a}=(-1)^a=1\]
and 
\[h^{(q)}_{q-1}(e_j)e_{ij}=e_{ij}\sum_{0\leq b\leq q-1} \binom{b}{0}\binom{e_j}{b}=e_{ij}h^{(q)}_0(e_j).\]
\end{proof}

\begin{lm}
The formula (2) is valid for $t=p^l$, where $0< l<r$.
\end{lm}
\begin{proof}
Assume $t=p^l$, where $0<l<r$. Then 
\[\binom{e_i+p^l}{k}=\sum_{\min\{0,k-p^l\}\leq b\leq k} \binom{p^l}{k-b}\binom{e_i}{b}\equiv \binom{e_i}{k}+\binom{e_i}{k-p^l} \pmod p\]
implies
\[\begin{aligned}&\sum_{a\leq k\leq q-1} (-1)^{k-a}\binom{k}{a}\binom{e_i+p^l}{k}\\
=&\sum_{a\leq k\leq q-1} (-1)^{k-a}\binom{k}{a}\binom{e_i}{k}+ \sum_{\max\{a,p^l\}\leq k\leq q-1} (-1)^{k-a}\binom{k}{a}\binom{e_i}{k-p^l}\\
=&\sum_{\max\{a-p^l,0\}\leq k<a} (-1)^{k+p^l-a}\binom{k+p^l}{a}\binom{e_i}{k}\\
&+\sum_{a\leq k\leq q-1-p^l} (-1)^{k-a}\Big[\binom{k}{a}-\binom{k+p^l}{a}\Big]\binom{e_i}{k}\\
&+\sum_{q-p^l\leq k\leq q-1} (-1)^{k-a}\binom{k}{a}\binom{e_i}{k}.
\end{aligned}\]

The congruence 
\[\binom{k+p^l}{a}\equiv \binom{k}{a}+\binom{k}{a-p^l} \pmod p\]
implies
$\binom{k+p^l}{a}\equiv \binom{k}{a-p^l} \pmod p$ for $k<a$.

If $a\geq p^l$ and $k\geq q-p^l$, then $k+p^l-a<q$, $a<q$ and $k+p^l\geq q$ implies $\binom{k+p^l}{a}\equiv 0 \pmod p$, and $\binom{k}{a}\equiv -\binom{k}{a-p^l} \pmod p$.

Therefore, if $a\geq p^l$, then 
\[\sum_{a\leq k\leq q-1} (-1)^{k-a}\binom{k}{a}\binom{e_i+p^l}{k}=
\sum_{a\leq k\leq q-1} (-1)^{k-a+p^l}\binom{k}{a-p^l}\binom{e_i}{k},\]
proving that $h_a^{(q)}(e_i)e_{ij}^{(p^l)}=e_{ij}^{(p^l)}h^{(q)}_{a-p^l}(e_i)$.

Now assume $a<p^l$. If $k<a$, then $\binom{k+p^l}{a}\equiv 0 \pmod p$. If $a\leq k\leq q-1-p^l$, then 
$\binom{k}{a}\equiv \binom{k+p^l}{a} \pmod p$. Finally, if $q-p^l\leq k\leq q-p^l-a$, then $\binom{k}{a}\equiv 0 \pmod p$. Therefore, 
\[\sum_{a\leq k\leq q-1} (-1)^{k-a}\binom{k}{a}\binom{e_i+p^l}{k}=
\sum_{q-p^l+a\leq k\leq q-1} (-1)^{k-a}\binom{k}{a}\binom{e_i}{k}=h^{(q)}_{q-p^l+a}(e_i)
\]
proving that $h^{(q)}_a(e_i)e_{ij}^{(p^l)}=e_{ij}^{(p^l)} h^{(q)}_{\overline{a-p^l}}(e_i)$.
\end{proof}

\begin{lm}
The formula (3) is valid for $t=p^l$, where $0< l<r$.
\end{lm}
\begin{proof} The congruence
\[\begin{aligned}&\binom{e_j-p^l}{k}=\sum_{0\leq b\leq k} (-1)^{k-b} \binom{p^l-1+k-b}{p^l-1}\binom{e_j}{b}\\
&\equiv 
\sum_{\substack{0\leq b\leq k\\b\equiv k \pmod {p^l}}}(-1)^{k-b}\binom{e_j}{b}\pmod p\end{aligned}\]
implies
\[\begin{aligned}&\sum_{a\leq k \leq q-1} (-1)^{k-a} \binom{k}{a}\binom{e_j-p^l}{k}\equiv
\sum_{a\leq k\leq q-1} (-1)^{k-a}\sum_{\substack{0\leq b\leq k\\b\equiv k\pmod{p^l}}}\binom{k}{a}\binom{e_j}{b}\\
&\equiv \sum_{0\leq b\leq q-1} (-1)^{a+b} \sum_{\substack{\max\{a-b,0\}\leq u\leq q-1-b\\u\equiv 0 \pmod{p^l}}}
\binom{b+u}{a} \binom{e_j}{b}\pmod p.
\end{aligned}\]
Since 
\[\binom{y+p^l}{a+p^l}=\sum_{a\leq s\leq a+p^l}\binom{p^l}{a+p^l-s}\binom{y}{s}\equiv \binom{y}{a}+\binom{y}{a+p^l} \pmod p,\]
we can apply this relation repeatedly to establish
\[\binom{b}{a+p^l}+\sum_{\substack{\max\{a-b,0\}\leq u\leq q-1-b\\u\equiv 0 \pmod{p^l}}}\binom{b+u}{a}\equiv \binom{b+wp^l}{a+p^l} \pmod p,\]
where $w$ is smallest integer such that $b+wp^l\geq q$. If $a+p^l<q$, then $\binom{b+wp^l}{a+p^l}\equiv 
0 \pmod p$ and 
\[\sum_{\substack{\max\{a-b,0\}\leq u\leq q-1-b\\u\equiv 0 \pmod{p^l}}}\binom{b+u}{a}\equiv -\binom{b}{a+p^l} \pmod p.\]
Since $\binom{b}{a+p^l}\equiv 0 \pmod p$ for $b<a+p^l$, for $a+p^l<q$
we obtain 
\[h^{(q)}_a(e_j)e_{ij}^{(p^l)}=e_{ij}^{(p^l)}\sum_{a+p^l\leq b\leq q-1} (-1)^{b-a-p^l}\binom{b}{a+p^l}\binom{e_j}{b} = e_{ij}^{(p^l)}h^{(q)}_{a+p^l}(e_j).\]
Assume $a+p^l\geq q$. In this case the sum 
\[\sum_{\substack{\max\{a-b,0\}\leq u\leq q-1-b\\u\equiv 0 \pmod{p^l}}}\binom{b+u}{a}\]
has at most one summand such that $b+u\geq a$. 

If $0\leq b<a+p^l-q<p^l$, then $\overline{b}=b<a+p^l-q=\overline{a}$ implies that 
\[\sum_{\substack{\max\{a-b,0\}\leq u\leq q-1-b\\u\equiv 0 \pmod{p^l}}}\binom{b+u}{a}=0.\]

If $b\geq a+p^l-q$, then 
\[\sum_{\substack{\max\{a-b,0\}\leq u\leq q-1-b\\u\equiv 0 \pmod{p^l}}}\binom{b+u}{a}=\binom{b+wp^l}{a},\]
where $w$ is the largest integer such that $b+wp^l<q$.
In this case, $a=\overline{a}+(q-p^l)$ and $b+wp^l=\overline{b}+(q-p^l)$.
Since $\binom{b+wp^l}{a}_j=\binom{b}{a+p^l-q}_j$ for each $j=0, \ldots, r-1$, we infer
$\binom{b+wp^l}{a}\equiv \binom{b}{a+p^l-q} \pmod p$.
Therefore,
\[h^{(q)}_a(e_j)e_{ij}^{(p^l)}=e_{ij}^{(p^l)}\sum_{a+p^l-q\leq b\leq q-1} (-1)^{b-a}\binom{b}{a+p^l-q}\binom{e_j}{b} = e_{ij}^{(p^l)}h^{(q)}_{a+p^l-q}(e_j).\]
\end{proof}

Since the commutation formulae are valid for each $t=p^l$, where $l=0, \ldots, r-1$, we can combine them and obtain the same formulae for every $1\leq t<q$.

\section{The generators of $\ker\pi_l$}

Finally, we describe the kernel of the morphism $\pi_l$.

Let $I_l$ denote the superideal of $Dist(G)$ generated by $J_l=Dist(T)\cap\ker\pi_l$. Let $V^{\pm }$ denote the largest unipotent subsupergroup of $B^{\pm}$. 
\begin{lm}\label{canonical form}
There is $I_l=Dist(V^+)Dist(V^-)J_l$.
\end{lm}
\begin{proof}
Since $Dist(G)=Dist(V^+)Dist(V^-)Dist(T)$ and $Dist(T)J_l =J_l$, 
it is sufficient to show that $Dist(G)J_l=J_l Dist(G)$, which is equivalent to $Dist(V^{\pm})J_l=
J_l Dist(V^{\pm})$. By Proposition \ref{commutation formulae}, there holds $h^{(q)}_{\alpha}(e)e_{ij}^{(t)}=e_{ij}^{(t)} h^{(q)}_{\beta}(e)$, where $\beta\equiv\alpha-t(\epsilon_i-\epsilon_j)\pmod q$.  Since $|\beta|\equiv |\alpha|\pmod q$, Lemma \ref{lm3.3} concludes the proof.  
\end{proof}

\begin{theorem}\label{the kernel}
There is $\ker\pi_l=I_l$.
\end{theorem}
\begin{proof}
We need to show that every  element $x=(\sum_s u^+_s u^-_s)h^{(q)}_{\alpha}(e)$, where $u^{\pm}_s\in Dist(V^{\pm})$ and $|\alpha|\equiv l\pmod q$, does not belong to $\ker\pi_l$. 

Recall that the natural $G$-supermodule $V$ has the basis $v_1, \ldots, v_{m+n}$, where the action of $e^{(t)}_{ij}\in Dist(G)$ on $V$ is given by $e_{ij}\cdot v_k=\delta_{jk} v_i$ and $e^{(t)}_{ij}\cdot v_k=0$ provided $t> 1$. For each $v_j,$ we define its \emph{height} $h(v_j)=j$, and for each 
$e_{ij}$, we define its \emph{height} $h(e_{ij})=i-j$. Then $h(e_{ij}v_j)=h(e_{ij})+h(v_j)$. In particular, $i>j$ implies 
$h(e_{ij}v_j)>h(v_j)$, and $i<j$ implies $h(e_{ij}v_j)<h(v_j)$.
Extending these definitions multiplicatively, for an element $z^+=\otimes_{i=1}^a v_{l_i}$, we denote its height $h(z^+)=\sum_{i=1}^a l_i$, and for  $b^+=\prod_{1\leq i<j\leq m+n} e_{ij}^{(t_{ij})}\in Dist(V^+)$, we denote its height $h(b^+)=\sum_{i<j} t_{ij}(i-j)$.
Denote by $W=V^*$ the dual of $V$ with the basis $w_1, \ldots, w_{m+n}$, where the action of $e^{(t)}_{ij}\in Dist(G)$ on $W$ is given by 
$e_{ij}\cdot w_k=-\delta_{ik}w_j$ and $e_{ij}^{(t)}\cdot w_k=0$ provided $t> 1$.
If $z^-=\otimes_{j=1}^b w_{l_j}$, then we define $h(z^+z^-)=h(z^+)$.

Let $u=\sum_s u^+_s u^-_s$, where $u^+_s\in Dist(V^+)$ and $u^-_s\in Dist(V^-)$ are monomial elements. Choose an index $s$ such that $h(u^+_s)=m_+$ is minimal possible. Denote $u^+_s=\prod_{1\leq i<j\leq m+n} e_{ij}^{(t_{ij})}$, $t^+=\sum_{i<j}t_{ij}$,
$u^-_s=\prod_{1\leq i>j\leq m+n} e_{ij}^{(t_{ij})}$ and $t^-=\sum_{i>j} t_{ij}$. 
Choose an integer $k\geq t^-$ such that $l+k-t^+\geq 0$, and nonnegative integers $\beta_i$ for $i=1, \ldots, m+n$ such that 
\[\beta_1+l-t^++t^-\equiv \alpha_1\pmod q,\] 
\[\beta_l+\sum_{i<l} t_{il}-\sum_{l>j}t_{lj} \equiv \alpha_j \pmod q \text{ for } l=2, \ldots, m+n,\]
and define 
\[z^+=v_1^{\otimes (l+k-t^+)}\otimes (\otimes_{j=1}^{m+n}v_j^{\otimes (\beta_j+\sum_{i<j}t_{ij})}), z^-=w_1^{\otimes (k-t^-)}\otimes (\otimes_{i=1}^{m+n}w_i^{\otimes (\sum_{i>j}t_{ij})}),\] and $z=z^+\otimes z^-$.
Then, by construction, $h^{(q)}_{\alpha}\cdot z=z$.
Define
\[S^+=v_1^{\otimes (\beta_1+l+k-t^+)}\otimes (\otimes_{i=1}^{m+n}(\otimes_{j=1}^{i-1} v_j^{\otimes t_{ji}}\otimes v_i^{\otimes\beta_i})),\]
\[S^-=w_1^{\otimes (k-t^-)}\otimes (\otimes_{i=2}^{m+n}(\otimes_{j=1}^{i-1} w_j^{\otimes t_{ij}})),\]
and $S=S^+\otimes S^-$.

The summand $S$ has the height 
$h(S)=h(S^+)=h(z^+)+m_+$, and it appears with a nonzero coefficient (equal to $\pm 1$) in $(u_s^+u^-_s)\cdot z$ only once as $u_S(z)$, where the endomorphism $u_S$ of the superspace
\[V^{\otimes (\beta_1+l+k-t^+ +\sum_{i=1}^{m+n}(\sum_{j=1}^{i-1} t_{ji} +\beta_i))}\otimes
W^{\otimes (k-t^- + \sum_{i=2}^{m+n}\sum_{j=1}^{i-1} t_{ij})}\]
is equal to
\[\mathrm{id}_V^{\otimes (\beta_1+l+k-t^+)}\otimes (\otimes_{i=2}^{m+n}(\otimes_{j=1}^{i-1} e_{ji}^{\otimes t_{ji}}\otimes \mathrm{id}_V^{\otimes\beta_i}))\otimes \mathrm{id}_W^{\otimes (k-t^-)}\otimes (\otimes_{i=2}^{m+n}(\otimes_{j=1}^{i-1} e_{ij}^{\otimes t_{ij}})).\]
If a nontrivial part of $u_s^-$ is applied to $z^+$, the height of the corresponding summand will increase.
Since the height $h(z^+)+m_+$ is the smallest height of all summands in $(u_s^+u_s^-)\cdot z$, every summand of this height appears only in $u_s^+\cdot z^+\otimes u_s^-\cdot z^-$.
 
We claim that $S$ does not appear in any other $(u^+_tu^-_t)\cdot z$ for $t\neq s$. If it does, then
$h(u^+_t)=m_+$, and $S$, as the summand of the smallest height, could appear only in $u^+_t\cdot z^+\otimes u^-_t\cdot z^-$. Then 
$S^+$ is a summand in $u^+_t\cdot z^+$ and $S^-$ is a summand in $u^-_t\cdot z^-$, which implies
$u^+_t=u^+_s$ and $u^-_t=u^-_s$, a contradiction. 

Therefere, $u h^{(q)}_{\alpha}\cdot z=u\cdot z\neq 0$, which shows that 
$ker(\pi_{l})=I_{l}$.
\end{proof}

\section{The case $char K=0$}

Assume now that $char K=0$.
\begin{lm}\label{lm3.3 in char=0}
The element $h_l=\sum_{1\leq i\leq m+n} e_i -l$ generates the ideal $J_l=\ker\pi_l\cap Dist(T)$.
\end{lm}
\begin{proof}
The algebra $Dist(T)$ is naturally isomorphic to the polynomial algebra $K[e_1, \ldots, e_{m+n}]$, freely generated by $e_1, \ldots, e_{m+n}$.
Moreover, if $W$ is a $G$-super-module and $w\in W_{\lambda}$, then the action of a polynomial $f=f(e_1, \ldots , e_{m+n})\in Dist(T)$ is given by  $f\cdot w=f(\lambda_1, \ldots, \lambda_{m+n})w$. Therefore, the element $h_l$ belongs to $\ker\pi_l\cap Dist(T)$. 

It remains to show that each $f\in\ker\pi_l\cap Dist(T)$ is divided by $h_l$. The algebra $Dist(T)$ is freely generated by $e_1, \ldots, e_{m+n-1}$ and $h_l$. In particular, modulo $Dist(T)h_l$, each polynomial $f$ is congruent to a polynomial $g=g(e_1, \ldots, e_{m+n-1})$, which depends on $e_1, \ldots, e_{m+n-1}$ only. Moreover, if $f$ belongs to $ker\pi_l\cap Dist(T)$, so does $g$. 

For any positive integers $N_1 > N_2>\ldots > N_{m+n-1}> l$ define 
\[\lambda=(N_1, N_2, \ldots, N_{m+n-1}, l-\sum_{1\leq i\leq m+n}N_i).\]
Then $\lambda$ belongs to $X(T)^+_l$, which implies $f(\lambda)=g(N_1, \ldots, N_{m+n-1})=0$. From here we conclude that $g=0$.
\end{proof}

As before, denote the superideal of $Dist(G)$ generated by $J_l=\ker\pi_l\cap Dist(T)$ by $I_l$.

\begin{theorem}\label{kernel0}
There is $ker \pi_l=I_l$.
\end{theorem}
\begin{proof}

By Lemma 7.7 from \cite{zubmarko}, the element $h_l$ is central in $Dist(G)$. Thus the superideal $I_l$ of $Dist(G)$, generated by $h_l$, equals $Dist(V^+)Dist(V^-)Dist(T)h_l$. The superalgebra
$Dist(G)/I_l$ has a basis consisting of all monomials
$u^+ u^- g$, where $u^{\pm}\in Dist(V^{\pm})$ and $g=g(e_1, \ldots, e_{m+n-1})$. 

Let $u=\sum_s u^+_s u^-_sg_s\in Dist(G)/I_l$, where $u^+_s\in Dist(V^+)$, $u^-_s\in Dist(V^-)$ are monomial elements and $g_s(e_1, \ldots, e_{m+n-1})\neq 0$. Choose an index $s$ such that $h(u^+_s)=m_+$ is minimal possible. Denote $u^+_s=\prod_{1\leq i<j\leq m+n} e_{ij}^{(t_{ij})}$, $t^+=\sum_{i<j}t_{ij}$,
$u^-_s=\prod_{1\leq i>j\leq m+n} e_{ij}^{(t_{ij})}$ and $t^-=\sum_{i>j} t_{ij}$. 
Fix an integer $k\geq \max\{t^-,t^+-l\}$.

Choose integers $N_1>\ldots >N_{m+n-1}>N_{n+m}= 0$ and define 
\[z^+=v_1^{\otimes (l+k-t^++N_1)}\otimes (\otimes_{j=1}^{m+n}v_j^{\otimes (\sum_{i<j}t_{ij}+N_i)}),\]
\[z^-=w_1^{\otimes (k-t^-)}\otimes (\otimes_{i=1}^{m+n-1}w_i^{\otimes (\sum_{i>j}t_{ij})})
\otimes w_{m+n}^{\otimes\sum_{m+n>j}t_{m+n,j}+\sum_{i=1}^{m+n} N_i},\] 
and $z=z^+\otimes z^-$.

Define
\[S^+=v_1^{\otimes (l+k-t^+)}\otimes (\otimes_{i=1}^{m+n}(\otimes_{j=1}^{i-1} v_j^{\otimes t_{ji}}\otimes v_i^{N_i})),\]
\[S^-=w_1^{\otimes (k-t^-)}\otimes (\otimes_{i=2}^{m+n}(\otimes_{j=1}^{i-1} w_j^{\otimes t_{ij}}))\otimes w_{m+n}^{\sum_{i=1}^{m+n} N_i},\]
and $S=S^+\otimes S^-$.

The weight of $z$ is
\[\begin{aligned}\mu=(&l-t^++t^- +N_1, \sum_{i<2}t_{i2}-\sum_{2<j}t_{2j}+N_2, \ldots,\\
& \sum_{i<a}t_{i,a}-\sum_{a<j}t_{a,j}+N_a,\ldots, \sum_{i<m+n}t_{i,m+n}-\sum_{i=1}^{m+n}N_i),
\end{aligned}\]
which implies $g_s\cdot z =g_s(\mu_1, \ldots, \mu_{m+n-1})z$.

Analogously as in the proof of Theorem \ref{the kernel}, we obtain that $u\in \ker \pi_l$ implies $g_s(\mu_1, \ldots, \mu_{m+n-1})=0$. Since the component $\mu_i$ 
vary by choice of $N_i$ for each $i=1, \ldots, m+n-1$, we conclude that $g_s=0$. This contradiction implies that $u\in I_l$.
\end{proof}

\end{document}